\documentclass[a4paper, 11pt]{amsart}
\usepackage{graphicx}
\graphicspath{ {disegni/} }

\usepackage{mleftright}
\usepackage{amsmath}
\usepackage{amsthm}
\usepackage{mathtools}
\usepackage{amssymb}
\usepackage{enumitem}
\usepackage{nicematrix}
\usepackage{tikz-cd}
\usepackage{import}
\usepackage{float}
\usepackage{caption}
\usepackage{subcaption}
\usepackage{pst-node}
\usepackage{tikz}
\usepackage{mdframed}
\usepackage{bm}
\usepackage[british]{babel}
\usepackage[useregional]{datetime2}
\DTMlangsetup[en-GB]{showdayofmonth=false}

\usepackage[colorlinks=true, allcolors=blue]{hyperref}
\usepackage[capitalise]{cleveref}
\usepackage[top=3cm,bottom=2cm,left=3cm,right=3cm,marginparwidth=1.75cm]{geometry}
\newtheorem{thm}{Theorem}
\renewcommand*{\thethm}{\Alph{thm}}

\newtheorem{theorem}{Theorem}[section]
\newtheorem{prop}[theorem]{Proposition}
\newtheorem{lemma}[theorem]{Lemma}
\newtheorem{corollary}[theorem]{Corollary}

\theoremstyle{definition}
\newtheorem{definition}[theorem]{Definition}
\theoremstyle{remark}
\newtheorem{remark}[theorem]{Remark}
\newcommand{\mtmathitem}{%
\xpatchcmd{\item}{\@inmatherr\item}{\relax\ifmmode$\fi}{}{\errmessage{Patching of \noexpand\item failed}}
\xapptocmd{\@item}{$}{}{\errmessage{appending to \noexpand\@item failed}}}

\newcommand{\R}{\mathbb{R}}
\newcommand{\D}{\mathbb{D}}
\renewcommand{\S}{\mathcal{S}}
\newcommand{\F}{\mathcal{F}}
\DeclareMathOperator{\dist}{dist}
\DeclareMathOperator{\id}{Id}
\DeclareMathOperator{\im}{Im}
\DeclareMathOperator{\diam}{diam}
\DeclareMathOperator{\codim}{codim}
\DeclareMathOperator{\rank}{rk}
\DeclareMathOperator{\sign}{sign}
\DeclareMathOperator{\crit}{crit}

\DeclarePairedDelimiter{\abs}{\lvert}{\rvert}
\makeatletter
\def\mathcenterto#1#2{\mathclap{\phantom{#1}\mathclap{#2}}\phantom{#1}}
\let\old@widetilde\widetilde
\def\widetildeto#1#2{\mathcenterto{#2}{\old@widetilde{\mathcenterto{#1}{#2\,}}}}
\let\old@widehat\widehat
\def\widehatto#1#2{\mathcenterto{#2}{\old@widehat{\mathcenterto{#1}{#2\,}}}}
\makeatother
\title{Thom--Milnor bounds for smooth manifolds}
\author{Saugata Basu, Antonio Lerario, Matteo Testa}

\date{\today}

\begin{document}  

\begin{abstract}
    We prove a smooth analogue of the classical Thom--Milnor bound, showing that the Betti numbers of the zero set of a smooth map on a compact Riemannian manifold can be controlled by a condition number computed from its first jet. This extends previous results in the Euclidean setting by Lerario and Stecconi [J. Singul., 2021]. As a key step, we generalize the Thom--Milnor bound to polynomial maps on a nonsingular real algebraic variety, improving the dependence on the degree. Finally, inspired by the work of B\"{u}rgisser, Cucker and Tonelli-Cueto [Found. Comput. Math., 2020], we introduce a condition number for families of functions. Using this we extend existing bounds due to Basu, Pollack and Roy [Proc. Amer. Math. Soc., 2004],
for the Betti numbers of semialgebraic sets  described by closed conditions to what we call closed semialgebraic type sets, namely sets defined by closed inequalities involving smooth functions.
\end{abstract}
\maketitle

\section{Introduction}

\subsection{A smooth Thom--Milnor bound}The classical Oleinik--Petrovsky--Thom--Milnor bound \cite{milnorBettiNumbersReal1964} provides an estimate for the Betti numbers of a real algebraic set. In particular, denoting by $b(\cdot)$ the sum of the Betti numbers of a topological space, if $Z(p)\coloneqq\{p=0\}\subseteq \R^n$ with $p$ a polynomial of degree $d$, the bound asserts that
\begin{equation}
\label{eq:tm}b(Z(p))=O(d^n).
\end{equation}

No regularity is assumed in the defining equation in the algebraic case. When we leave the algebraic setting, however, to control the topology of the zero set of  a smooth function one needs some additional regularity conditions, since any closed set can be realized as the zero set of such a function.
If the function's domain is the $n$--dimensional disk and its zero set is regular, this problem was first studied by Y. Yomdin in \cite{Yomdinglobal}. He obtained a bound on the Betti numbers of the function's zero set in terms of the decay rate of its Taylor coefficients. For polynomials, whose Taylor coefficients eventually vanish, this bound recovers \eqref{eq:tm}. In a similar direction, the second named author and M. Stecconi, in \cite{lerarioWhatDegreeSmooth2021}, studied smooth systems on the $n$--dimensional disk that can be approximated by polynomial systems without affecting the topology of their solution set. In \cite{lerarioWhatDegreeSmooth2021} the authors introduced a notion of ``condition number'' for such systems, proving a bound on the sum of the Betti numbers of their set of solutions that is polynomial of degree $n$ in this quantity.

In this paper, we extend these results to the case where the ambient manifold is not Euclidean, but rather a compact Riemannian manifold $(M,g)$ of dimension $m$. We first introduce the set of functions that have at least a singular zero, called the \emph{discriminant} $\Delta$ see \cref{def:discriminant}:
\[\Delta:=\left\{f\in C^1(M, \R^k)\,\bigg|\, \exists x\in M\,:\, f(x)=0,\, \mathrm{rk}(D_xf)< k\right\}\subseteq C^1(M, \R^k).\]  Our result shows that the $C^1$--norm of $f$, computed using $g$ (see \cref{C^1 norm}), together with the $C^1$ distance from the discriminant controls the sum of the Betti numbers of $Z(f)$. \begin{thm}\label{theorem A}
    Let $(M,g)$ be a compact Riemannian manifold of dimension $m$. There exists a constant $c_1=c_1(M,g)>0$ such that for any \mbox{$f\in C^1(M,\R^k)$}, we have: 
    \[b(Z(f))\le c_1\cdot \left(\frac{\|f\|_{C^1}}{\dist_{C^1}(f,\Delta)}\right)^m.\]
\end{thm} 
To simplify the notation we denote by $$\delta(f)\coloneqq \dist_{C^1}(f,\Delta)\quad \textrm{and} \quad k(f)\coloneqq\frac{\|f\|_{C^1}}{\dist_{C^1}(f,\Delta)}.$$ We call $k(f)$ the \emph{condition number} of $f$. The name comes from the fact that $\delta(f)$ can be effectively computed as the minimum of the fiberwise distance of the first jet of $f$ to the set of singular jets, see \cref{equality distance discriminant}.

We will show in \cref{propo:sharp} that the order $O(k(f)^m)$ is sharp. For what concerns the constant $c_1(M, g)$ from the statement, we will show in \cref{constant dependence} that it can be chosen to depend only on the diameter, sectional curvature and volume of the manifold.

\subsection{Real algebraic geometry}A key step in the proof of \Cref{theorem A} consists in a generalization of the Oleinik--Petrovsky--Thom--Milnor bound, which is of independent interest. In particular, assume that $M=Z(q_1,\dots,q_\ell)\subseteq \R^n$ is a compact, regular algebraic variety of dimension $m$, where each $q_i$ has degree at most $d_0$, and that $p\colon \R^n\to \R^k$ is a polynomial map of degree $d$. We show that the bound on the sum of Betti numbers of $b(Z(p)\cap M)$ can be improved from $O(d^n)$ to $O(d^m)$. More precisely:
\begin{thm} \label{theorem B}
    Let $M=Z(q_1,\dots, q_\ell)$ be a regular, irreducible compact manifold of dimension $m$, with $\deg q_i\le d_0$. For $d\ge 1$, consider $p_1,\dots, p_k$ polynomials with $\deg p_i\le d$, then:
    \[b(Z(p_1,\dots,p_k)\cap M)\le d_0^{n-m}((n-m)(d_0-1) + 2d-1)^{m}\leq O(d^m).\]
\end{thm}
The proof of \Cref{theorem A} then consists in approximating both the map $f$ and the manifold $M$ by algebraic counterparts, therefore reducing the problem to the case of \Cref{theorem B}.

\subsection{Semialgebraic--type sets on a smooth manifold}\label{sat}The Oleinik--Petrovsky--Thom--Milnor bound was generalized in the context of semialgebraic geometry to different kinds of sets described by polynomial equations and inequalities; see for example \cite{basuAlgorithmsRealAlgebraic2006}. The most general objects are semialgebraic sets. A semialgebraic set on a set of polynomials $p_1,\dots, p_s\in \R[x_1,\dots,x_n]$ is a set of the form  
\[
\bigcup_i\bigcap_{j=1}^s \{ p_j*_{ij}0\},\]
where $*_{ij}\in \{<,>,=\}$. 

In the same spirit, on a Riemannian manifold $(M,g)$, we say that $S\subseteq M$ is a set of semialgebraic--type on a family $\F=\{f_1,\dots,f_s\}$, where $f_i\in C^1(M,\R)$, if
\[
S\coloneqq\bigcup_i\bigcap_{j=1}^s \{ f_j*_{ij}0\},\]
where $*_{ij}\in \{<,>,=\}$. By introducing a condition number $k(\F)$ for finite families, we generalize \cref{theorem A} to obtain a bound for the sum of Betti numbers of closed semialgebraic type set. The bound, inspired by \cite{basuBettiNumbersSign2004}, is polynomial of order $m$ in both the condition number and the number $s$ of functions in the family:
\begin{thm} \label{theorem C}
    Let $(M,g)$ be a smooth Riemannian manifold of dimension $m$. There exists a constant $c_4(M,g)>0$ such that for every smooth family $\F$, and every closed semialgebraic--type set $S$ on the family $\F$, we have  
    \[b(S)\le c_4 \cdot s^m k(\F)^m .\]
\end{thm}
The case when $S$ is not closed is more involved. In the classical semialgebraic setting, a technique due to Gabrielov and Vorobjov, see \cite{gabrielovApproximationDefinableSets2009}, allows one to approximate a non--closed semialgebraic set by a closed one having the same Betti numbers. As a consequence for a semialgebraic set $S\subseteq\mathbb{R}^n$ described by $s$ polynomials of degree less than $d$, we have $$b(S)\le O(s^nd^n).$$ Gabrielov and Vorobjov's work, however, does not treat semialgebraic sets restricted to a variety $M$ of dimension $m$, so their argument cannot be applied directly to obtain a bound of the form $O\bigl(s^{m}k(\F)^{m}\bigr)$ for the sum of Betti numbers of semialgebraic type sets. Nevertheless, the fact that the variety $M$ is nonsingular in our setting should simplify the argument, allowing one to prove an analogous version of the Gabrielov--Vorobjov theorem. We plan to analyze this case in a forthcoming work.


\section{Preliminaries}
\subsection{Polynomial approximations}
\subsubsection{Jets and $C^1$--norms}
The usual notion of $C^1$--norm for functions $f\colon M\to \R^k$ relies on the notion of jet bundle. We briefly recall these notions here.
For further details, we refer to \cite{hirschDifferentialTopology1976} and \cite{saundersGeometryJetBundles1989}. In everything that follows we will assume $M$ to be a smooth compact manifold.
Let us denote  the space of $C^1$ maps $f\colon M\to \R^k$ by $C^1(M,\R^k)$. Moreover we will denote the zero set of a function $f=(f_1, \ldots, f_k):M\to \R^k$ by 
\[Z(f)\coloneqq\left\{x\in M\,\bigg|\, f_1(x)=\cdots =f_k(x)=0\right\}.\]

Two functions $f,g\in C^1(M,\R^k)$ are said to be $1$--equivalent at $x\in M$ if for some local parametrization $\varphi \colon U \to M$, $x=\varphi(y)\in \varphi(U)$, we have that:
\[\frac{\partial f\circ \varphi}{\partial y^i}(y)= \frac{\partial g\circ \varphi}{\partial y^i}(y), \]
for any $i=1,\dots, m$. This is an equivalence relation, and the equivalence class of $f$ under this relation is called  the first jet of $f$ at $x$ and is denoted by $j^1_xf$. The set of all first jets at $x$ is denoted by $J^1_x(M,\R^k)$.
\begin{definition}[First jet bundle] The first jet bundle is defined as the set 
\[J^1(M,\R^k)\coloneqq \{j^1_xf\mid x\in M, f\in C^1(M,\R^k)\}.\]
\end{definition}
The manifold $J^1(M, \R^k)$ is a vector bundle: there is a natural projection map
\[\pi \colon J^1(M,\R^k)\to M,\quad j^1_xf\mapsto x.\] 
This map makes the triple $(J^1(M,\R^k),\pi, M)$ a vector bundle. Given $f\in C^1(M,\R^k)$, there is a corresponding section of the jet bundle $j^1f\colon x\mapsto j^1 f(x)=j^1_xf$, called the first jet extension of $f$. The  first jet bundle $J^1(M,\R^k)$ can be shown to be isomorphic to $\R^k\times T^*M$, and thus each element of $J^1(M,\R^k)$ can be identified by a triple $(x,f(x),D_xf)$ for some $f\in C^1(M,\R^k)$. 
\begin{definition}[Fiberwise norm on $J^1(M, \R^k)$]\label{C^1 norm}
    Let $\|\cdot\|\colon J^1(M,\R^k)\to\R$
    be a smooth function such that, when restricted to any fiber $J^1_x(M,\R^k)$, it defines a norm. Then, for any $f\in C^1(M,\R^k)$, we define its $C^1$--norm as $$\|f\|_{C^1(M,\R^k)}\coloneqq \max_{x\in M}\|j^1f(x)\|.$$
\end{definition}
Note that the maximum is attained by the compactness of $M$.

\begin{remark}
   On each fiber \( J^1_x(M, \mathbb{R}) \), any two norms \( \|\cdot\|^1_x \) and \( \|\cdot\|^2_x \) are equivalent, as all norms on a finite-dimensional vector space are equivalent. Moreover, since \( M \) is compact and these norms vary continuously with \( x \), it follows that any two \( C^1 \)-norms \( \|f\|^1_{C^1(M, \mathbb{R}^k)} \) and \( \|f\|^2_{C^1(M, \mathbb{R}^k)} \) are also equivalent. The topology on \( C^1(M, \mathbb{R}^k) \) induced by any of these norms is called the Whitney strong topology.
\end{remark}

 If $(M,g)$ is a smooth Riemannian manifold, we can define a fiberwise norm on $J^1(M,\R^k)$ in a canonical way by using the Riemannian metric and the standard scalar product in $\R^k$. 
 We set $$\|j^1f(x)\|\coloneqq \|f(x)\|+\|D_xf\|.$$ 
 Indeed, $\|D_xf\|\coloneqq \max_{\|v\|=1}\|D_xf(v)\|,$ where the norm of $v\in T_xM$ is computed using the Riemannian structure. 
 
   This is the fiberwise norm we will use for the rest of the paper.
\subsubsection{Quantitative Weierstrass approximation}
The first ingredient in the proof of \cref{theorem A} is a quantitative version of the Stone--Weierstrass theorem. It enables us to approximate a $C^1$--function by a polynomial one, with the approximation error measured by the $C^1$--norm of the function. 
\Cref{bagbi}, presented below, provides a general statement in this direction and can be found in \cite{bagbyMultivariateSimultaneousApproximation2002}.
\begin{theorem}[Quantitative Weierstrass approximation]    
\label{bagbi}
    Let $M\subseteq \R^n$ be a compact set with the property that any two points $a,b\in M$ can be joined by a rectifiable curve in $M$ whose length is  $O(|a-b|)$.
    Let $f\in C^1(U,\R^k)$ where $U$ is an open neighborhood of $M$. Then, there exists a constant $c_0=c_0(M)>0$ such that for each  $d \ge 1$ there is a polynomial map $p=(p_1,\dots, p_k)$, where each $p_i$ is of degree $\le d$, which satisfies:
    \begin{equation}\label{bagbi eq}
        \max_{x\in M} \|f(x)- p(x)\|\le \frac{c_0}{d} \max_{x\in M}\big(\|f(x)\|+{\max_{\substack{\|v\|=1 \\ v\in \R^n}}\|D_xfv\|} \big).
    \end{equation}
\end{theorem} 
Notice that the right-hand side of \eqref{bagbi eq} resembles the $C^1$--norm of $f$ given in \Cref{C^1 norm}, i.e. $$\|f\|_{C^1(M,\R^k)}=\max_{x\in M}\|j^1f(x)\|=\max_{x\in M}(\|f(x)\|+\|D_xf\|),$$ computed using the fiberwise norm on $J^1_x(M,\R^k)$ induced by the Riemannian metric $i^*g_{\R^n}$.
However, the two quantities are actually distinct.
 Indeed, in \Cref{bagbi}, $f$ is defined on an open neighborhood $U$ of $M$, and the quantity $ D_xf(v)$ is maximized over all $v\in \R^n$ such that $\|v\|=1$. On the other hand, by  \Cref{C^1 norm},
$$\|j^1f(x)\|=\|f(x)\|+\|D_xf\|_{i^*g_{\R^n}}=\|f(x)\|+\max_{\substack{\|v\|=1 \\ v\in T_xM}} D_xfv,$$ where we see that $D_xf(v)$ is maximized only for $v\in T_xM$ such that $\|v\|=1$. Thus, a priori, $\|j^1f\|$ is a smaller quantity.  In fact, the next proposition shows that, if $M$ is a compact manifold, in \Cref{bagbi} it is sufficient to consider $\|f\|_{C^1(M,\mathbb{R}^k)}$. In this case, we also require $f$ to be defined only on $M$ and not on an open neighborhood $U$.

For the proof of next result, we recall from \cite{federerCurvatureMeasures1959} the notion of \emph{reach} of an embedded manifold $M\subseteq \R^n$: it is the largest $r>0$ such that each point in the set $\{x\in \R^n\mid \dist(x,M)<r\}$ has a unique nearest point in $M$.

\begin{prop} \label{weierstrass appro to use}
    Let $M$ be a smooth compact submanifold of $\R^n$. Let $c_0$ be the constant provided by \Cref{bagbi}, then, for any  $f\in C^1(M,\R^k)$ and any integer $d \ge 1$, there exists a polynomial map $p=(p_1,\dots,p_k)$,  where each $p_i$ is of degree $\le d$, which satisfies:
    \[\|f-p\|_{C^0(M,\R^k)}\le \frac{c_0}{d}\|f\|_{C^1(M,\R^k)}.\]
\end{prop}
\begin{proof}
To deduce \Cref{weierstrass appro to use} from \Cref{bagbi}, we need to establish two things:
\begin{enumerate}
    \item The existence of a constant $\sigma > 0$ such that for each pair of points $a,b\in M$, there exists a smooth curve $\gamma$ connecting $a$ and $b$ whose length is $\sigma |a-b|$.
    \item An extension of $f \in C^1(M, \mathbb{R}^k)$ to a function $g \in C^1(U, \mathbb{R}^k)$ defined on an open neighborhood $U$ of $M$, whose $C^1$--norm coincides with that of $f$:
    \begin{equation*} 
        \|f\|_{C^1(M, \mathbb{R}^k)} = \max_{x \in M} \left( \|g(x)\| + \|D_x g\| \right).
    \end{equation*}
\end{enumerate}
\textit{Existence of $\sigma$:} Consider a positive radius $r$ such that $\text{reach}(M) > r > 0$. Let $U$ be the closed $r$-tubular neighborhood of $M$ defined as
\[
    U \coloneqq \{ x \in \mathbb{R}^n \mid \text{dist}(x, M) \leq r \}.
\]
If $x, y \in M$ satisfy $\|x - y\| < r$, then the line segment $\gamma(t) = x + t(y - x)$ for $t \in [0, 1]$ is contained in $U$. Since $U$ is compact, the nearest point projection $\pi \colon U \to M$ has a bounded differential. Let $\sigma_1 > 0$ be such that
\[
    \|D_x \pi\| \leq \sigma_1 \quad \text{for all } x \in U.
\]
Then,
\[
    \text{dist}_M(x, y) \leq L(\pi \circ \gamma) \leq \sigma_1 \cdot L(\gamma) = \sigma_1 \|x - y\|,
\]
where $L(\cdot)$ denotes the length of the curve.
If $\|x - y\| \geq r$, by compactness of $M$, we can define
    \[
        \sigma_2 \coloneqq \max_{\substack{x, y \in M \\ \|x - y\| \geq r}} \frac{\text{dist}_M(x, y)}{\|x - y\|}.
    \]
Set $\sigma \coloneqq \max\{\sigma_1, \sigma_2\} > 0$. This ensures that for any $x, y \in M$, 
    \[
        \text{dist}_M(x, y) \leq \sigma \|x - y\|.
    \]
\textit{Existence of extension $g$:}  For any $x\in U$, define $g(x)=f(\pi(x))$.
If $x \in M$, we can decompose the tangent space as
    \[
        T_x \mathbb{R}^n = T_x M \oplus N_x M,
    \]
where $N_x M = \ker D_x \pi$. Since
    $
        D_x g = D_{\pi(x)} f \circ D_x \pi,
    $
and $D_x \pi$ vanishes on $N_x M$, we have
    \begin{align*}  
    \max_{x \in M} \left( \|g(x)\| + \|D_x g\| \right) &= \max_{x \in M} \biggl( \|g(x)\| + \max_{\substack{v \in T_x M \oplus N_x M\\ \|v\|=1}} \|D_x g(v)\| \biggr) \\
    &= \max_{x \in M} \biggl( \|f(x)\| + \max_{\substack{v \in T_x M \\ \|v\|=1}} \|D_x f(v)\| \biggr) \\
    &= \|f\|_{C^1(M, \mathbb{R}^k)}. \phantom{\biggl( }
    \end{align*}    
Now \Cref{weierstrass appro to use} follows directly by \Cref{bagbi}.
\end{proof}

\subsubsection{Distance from discriminant}
We introduce now the notion of discriminant in the space of $C^1$--functions. The discriminant is defined as the set of functions whose zero sets are not regular, and the distance from this discriminant is used to quantify the regularity of a function.

We recall some elementary definitions. For $f\in C^1(M,\R^k)$, we say that $f$ is transversal to a submanifold $Y\subseteq \R^k$, and denote it by $f\pitchfork Y$, if for every $x\in M$ such that $f(x)\in Y$, we have 
\[\im D_xf+T_{f(x)}Y=T_{f(x)}\R^k.\]
In particular, when $Y=\{0\}$, if $f\pitchfork \{0\}$ we say that $0$ is a regular value. Equivalently, $0$ is a regular value if for every $x\in M$ such that $f(x)=0$ we have that $\rank D_x f=k$. If $M$ is of dimension $m$ and $m<k$, then this implies that $f^{-1}(0)=\emptyset$. Since this case is trivial, we will always implicitly assume that $m\ge k$. The following definitions work also in the case when $M$ is not compact. In case $M$ is compact,  every infimum is attained at some point and so can be replaced by a minimum. 
\begin{definition}[Discriminant]\label{def:discriminant}
    We define the discriminant $\Delta \subseteq C^1(M,\R^k)$ as
    \[\Delta\coloneqq \left\{f\in C^1(M,\R^k)\,\bigg|\, \exists x\in M:\, f(x)=0, \, \mathrm{rk}(D_xf)<k \right\}.\]
    For any $f\in C^1(M,\R^k)$ we define its distance from the discriminant as 
    \[{\delta(f)}:=\dist_{C^1}(f,\Delta)=\inf_{g\in \Delta} \|f-g\|_{C^1(M,\R^k)}\]
\end{definition}
Since $\Delta$ belongs to an infinite--dimensional space, it is more convenient, when computing distances, to consider a finite dimensional object sitting in the space of jets which captures the same information.
\begin{definition}\label{definition Sigma}
    Let $\Sigma_x\subseteq J^1_x(M,\R^k)$:
    \[\Sigma_x\coloneqq \bigg\{j^1f(x)\;\bigg|\; f(x)=0, \rank D_xf<k\bigg\}\]
    We define $\Sigma\coloneqq \bigcup_{x\in M}\Sigma_x\subseteq J^1(M,\R^k)$.
\end{definition}


Next we show that in the case $\|\cdot\|_{C^1}$ is induced by a Riemannian metric, $\delta(f)$ can be computed in terms of the singular values of $D_xf$. Given the scalar product $g_x$ on $T_xM$ and the standard scalar product on $\R^k$, the singular values of $D_xf$ are defined by:
\[\sigma_i(D_xf)\coloneqq\min_{\substack{V\subseteq T_xM \\ \dim V = m-i+1}}\max_{\substack{v\in V\\ \|v\|_{g_x}=1}}\|D_xf(v)\|\]
It is clear that $\sigma_1\ge \dots \ge\sigma_k\ge 0 $ are the only non zero singular values. Moreover, {note that} we have $j^1f(x)\in \Sigma_x$ if and only if $f(x)=0$ and $\sigma_k(D_xf)=0$. 
\begin{lemma}\label{lem: distance discriminant = distance sigma} Let $(M,g)$ be a Riemannian manifold. Then
    $$\dist(j^1f(x),\Sigma_x)= \|f(x)\|+ \sigma_k(D_x f).$$
\end{lemma}
\begin{proof}
    By the definition of $\Sigma_x$, we have that
    \begin{align} \label{singular value}
        \dist(j^1f(x),\Sigma_x)&= \inf_{j^1g(x)\in \Sigma_x} \|j^1f(x)-j^1g(x)\| \nonumber \\
        &= \inf_{j^1g(x)\in \Sigma_x} \big(\|f(x)\|+\max_{v\in T_x^1M} (D_xf(v)-D_xg(v))\big)
    \end{align} 
    Let $V\subseteq T_xM$ be a $(k-1)$--dimensional subspace. Consider $g_V\in C^1(M,\R^k)$ such that $g_V(x)=0$,  $D_xg_V(v)=D_xf(v)$ for each  $v\in V$ and $D_xg_V(v)=0$ for each $v\in V^\perp$. For example, we can assume $g_V$ to be an affine function. Note that $j^1g_V(x)\in \Sigma_x$ because $\dim (\ker D_xg_V)=m-k+1$. 
    
    Taking the infimum in (\ref{singular value}) over such $g_V$, we obtain
    \begin{align*}
        \dist(j^1f(x),\Sigma_x)&\le\|f(x)\|+\inf_{\substack{V\in T_xM \\ \dim V=k-1}} \max_{\substack{v\in V^\perp\\ \|v\|=1}}\|D_xf(v)\|\\
        &=\|f(x)\|+\inf_{\substack{V^\perp\in T_xM \\ \dim V^\perp=m-k+1}} \max_{\substack{v\in V^\perp\\ \|v\|=1}}\|D_xf(v)\|\\
        &= \|f(x)\|+\sigma_k(D_xf)
    \end{align*}
    The opposite inequality follows from \eqref{singular value} by noticing that $\dim \ker D_xg\ge m-k+1$, indeed \begin{align*}
        \max_{v\in T_x^1M}\|D_xf(v)-D_xg(v)\| &\ge \max_{v\in (\ker D_xg )}\|D_xf(v)\| \\
        &\ge \sigma_k(D_xf).
    \end{align*}
\end{proof}

The next Proposition states that the distance from the discriminant $\Delta$ in the infinite--dimensional space $C^1(M,\R^k)$ coincides with the minimum of the fiberwise distance of $j^1f$ to the set of singular jets.
\begin{prop}\label{equality distance discriminant}
    The following equality holds:
    $$\delta(f)={\inf_{x\in M}\dist(j^1f(x),\Sigma_x)}.$$
\end{prop}
For the proof, we need the following elementary lemma.
\begin{lemma} \label{gaussian}
    Given $\sigma=(\sigma_0,\sigma_1)\in J^1_0(\R^m,\R^k)$ and any $\epsilon,r>0$,  there exists $\lambda>0$ such that the function
    \[s_\lambda(v)\coloneqq e^{-\|v\|^2\lambda}\sigma_1(v)+\sigma_0\]
    satisfies: 
    \begin{enumerate}    
        \item[$(a)$] $j^1s_\lambda(0)=\sigma$;
        \item[$(b)$]  $\sup_{v\in \R^m}\|s_\lambda(v)\|=\|\sigma_0\|+\epsilon$;
        \item[$(c)$]   $\sup_{v\in \R^m}\|D_vs_\lambda\|=\|\sigma_1\|+\epsilon$;
        \item[$(d)$]   $\sup_{\|v\|\ge r}\|D_vs_\lambda\|\le \epsilon$.
    \end{enumerate}
\end{lemma}
\begin{proof}
    Properties $(a)$, $(b)$ and $(d)$ are straightforward. For $(c)$, we compute 
    \begin{align*}
        \sup_{\|w\|=1}\|D_v (e^{-\|v\|^2\lambda}\sigma_1(v)+\sigma_0)(w)\|&= \sup_{\|w\|=1}\|e^{-\|v\|^2\lambda}(\sigma_1(w)-2\lambda \langle v,w\rangle \sigma_1(v))\|\\
        &\le \sup_{\|w\|=1} e^{-\|v\|^2\lambda} \|\sigma_1\|(\|w\|+2\lambda \|w\|\|v\|^2)
    \end{align*}
    Since $e^{-\|v\|^2\lambda}2\lambda \|w\|\|v\|^2\to 0$ uniformly, we can choose $\lambda>0$ sufficiently large so that $$\sup_{\|w\|=1}\|D_v s_\lambda(w)\|\le \|\sigma_1\|+\epsilon.$$
\end{proof}

\begin{proof}[Proof of \Cref{equality distance discriminant}]
    Denote by $$\Delta_x\coloneqq\{g\in C^1(M,\R^k)\mid g\not \pitchfork 0 \text{ at } x\}\subseteq C^1(M,\R^k)$$
    {and by
    \[\widetilde{\delta}(f):=\inf_{x\in M}\dist(j^1f(x),\Sigma_x).\]
    We need to show that $\widetilde{\delta}(f)=\delta(f)=\mathrm{dist}(f, \Delta)$.}
    
    First, we show that $\widetilde{\delta}(f)\le\dist_{C^1}(f,\Delta)$:
    \begin{align*}
        \dist_{C^1}(f,\Delta)&=\inf_{g\in \Delta}\|f-g\|_{C^1(M,\R^k)}\\
        &=\inf_{x\in M}\inf_{g\in \Delta_x}\|f-g\|_{C^1(M,\R^k)}\\
        &\ge \inf_{x\in M}\inf_{g\in \Delta_x}\|j^1f(x)-j^1g(x)\|\\
        &=\inf_{x\in M} \dist(j^1f(x), \Sigma_x)\\
        &=\widetilde{\delta}(f).
    \end{align*}
For the opposite inequality, let $x_0\in M$ and consider: 
\begin{align}
    \inf_{g\in \Delta_{x_0}}\|f-g\|_{C^1(M,\R^k)}&=
    \inf_{\phantom{lj^1+} j^1g({x_0})\in \Sigma_{x_0}\phantom{h(x)}}\|f-g\|_{C^1(M,\R^k)}\nonumber\\
    &=\inf_{j^1h({x_0})\in \Sigma_{x_0}+j^1f({x_0})}\|h\|_{C^1(M,\R^k)}\ \nonumber\\
    &=\inf_{\phantom{ldd}\eta\in \Sigma_{x_0}+j^1f({x_0})\phantom{d+}}\hspace{-10pt}\inf_{j^1h({x_0})=\eta}\|h\|_{C^1(M,\R^k)}. \label{difficult part of distance discri}
\end{align}
Given a jet $\eta=(\eta_0,\eta_1)\in \Delta_{x_0}+j^1f({x_0})$, let $\psi\colon \R^m\to U\subseteq M$ be a local parametrization centered at ${x_0}$ and denote by $\sigma=(\sigma_0,\sigma_1)\coloneqq\psi^*\eta\in J_0(\R^m,\R^k)$ the jet in local coordinates. By \Cref{gaussian}, for any $r,\epsilon>0$, there exists $\lambda>0$ and $s_\lambda\in C^1(\R^m,\R^k)$ such that $j^1s_\lambda(0)=\sigma$.
Set
\[
h_\lambda(x)\coloneqq
\begin{cases}
    s_\lambda\circ \psi^{-1}(x) & \text{for } x\in U\\
    0 & \text{for } x\not\in U.
\end{cases}\]
 Then $j^1h_\lambda({x_0})=(\psi^{-1})^*\psi^*\eta=\eta$. Furthermore, by property $(b)$ of \Cref{gaussian},
\[\sup_{x\in U}\|h_\lambda(x)\|=\sup_{v\in \R^m}\|s_\lambda (v)\| =\|\sigma_0\|+\epsilon=\|\eta_0\|+\epsilon.\]
On $\R^m$ consider a constant metric equal to $\psi^*g(0)$  and let  $$\sup_{x\in U}\|D_xh_\lambda\|\le\sup_{x\in U} \|D_{\psi^{-1}(x)}s_\lambda\|\|D_x\psi^{-1}\|.$$ By continuity, for sufficiently small $r$, we have $ 1-\epsilon\le\|D_x\psi^{-1}\|\le 1+\epsilon$ for each $x\in \psi(B_r(0))$  and by property $(c)$ of \Cref{gaussian},
\[\sup_{x\in \psi(B_r(0))} \|D_{\psi^{-1}(x)}s_\lambda\|\|D_x\psi^{-1}\|\le \|\eta_1\|(1+\epsilon).\]
Also by property $(d)$, 
\[\sup_{x\in \psi(\R^m\setminus B_r(0))} \|D_{\psi^{-1}(x)}s_\lambda\|\|D_x\psi^{-1}\|\le \epsilon \cdot \sup_{x\in \psi(\R^m\setminus B_r(0))}  \|D_x\psi^{-1}\|\]
Thus (\ref{difficult part of distance discri}) becomes:
\begin{align*}
    \inf_{g\in \Delta_{x_0}}\|f-g\|_{C^1(M,\R^k)}&=\hspace{-10pt}\inf_{\phantom{dd}\eta\in \Sigma_{x_0}+j^1f({x_0})\phantom{d+}}\hspace{-10pt}\inf_{j^1h({x_0})=\eta}\|h\|_{C^1(M,\R^k)} \\    &\le\hspace{-10pt}\inf_{\phantom{dd}\eta\in \Sigma_{x_0}+j^1f({x_0})\phantom{d+}}\hspace{-10pt} \|\eta\|+o(\epsilon)
\end{align*}
Taking the infimum over $x_0\in M$ on both sides and letting $\epsilon\to 0$ we conclude that $\dist_{C^1}(f,\Sigma)\le \widetilde{\delta}(f).$
\end{proof}
The distance from the discriminant, however, is not an homogeneous quantity. In particular, we can multiply $f$ by a constant $\lambda>1$ and increase its distance from the discriminant. Therefore, we introduce the following.
\begin{definition}[Condition number]\label{def:condition}
    Let $f\in C^1(M,\R^k)$, the condition number $k(f)$ is defined by
    $$k(f)\coloneqq \frac{\|f\|_{C^1}}{\delta(f)}.$$
\end{definition}
The condition number can be considered as a notion {quantifying the} regularity of a zero set $Z(f)$. 
\subsubsection{Stability of Betti numbers}
Since the general idea for the proof of our theorems is to approximate the functions by polynomials, in this section we study how approximation in different norms affect the topology of the zero sets.

We begin by {recalling how regular zero sets are stable under $C^1$--perturbations}. In particular, under small $C^1$--perturbations, the Betti numbers are continuous, and therefore locally constant. The {next proposition makes this more precise}: it is a direct corollary of Thom's first Isotopy Lemma, see \Cref{thom isotopy particular}, but a more elementary proof can be found in \cite[Lemma 2.4]{LerarioLectureNotes} .
\begin{prop} \label{stability lemma}
    Let $(M,g)$ be a compact Riemannian manifold and let  $f\in C^1(M,\R^k)$. Consider $g\in C^1(M,\R^k)$ satisfying \[\|f-g\|_{C^1(M,\R^k)}<\delta(f),\]
    then $(M,Z(f))$ is isotopic to $ (M,Z(g))$ and also $(M,\{f\le 0\})$ is isotopic to $(M,\{g\le 0\})$.
\end{prop}
By \Cref{weierstrass appro to use}, any function $f$ with bounded $C^1$--norm can be quantitatively approximated in the $C^0$--topology by a polynomial map $p$. In general, the topology of the zero set is not stable under $C^0$--perturbations, so there is no immediate relationship between the Betti numbers of $Z(f)$ and $Z(p)$. However, if $f$ is regular, the Betti numbers of its zero set enjoy a semicontinuity property under $C^0$--perturbations. Indeed, we will show in \Cref{semicontinuity}, that $b(Z(f))\le b(Z(p))$. The result follows directly from the following general statement which can be found in \cite[Theorem 179]{randomdifferentialtopology}. The theorem provides a lower semicontinuity property for the homology of a regular preimage $f^{-1}(Y)$ under $C^0$ perturbation of the map $f$. Notice that the use of $\check{H}^i$, denoting the $i$--th \v{C}ech cohomology group, is necessary since there are counterexamples for singular cohomology.

\begin{theorem}\label{general semicontinuity}
    Let $M,N$ be smooth manifolds, let $Y\subseteq N$ be a closed and cooriented smooth submanifold. Let $f\in C^1(M,N)$ such that $f\pitchfork Y$ and $Z\coloneqq f^{-1}(Y)$. Let $U,U_1$ be two open tubular neighborhoods of $Z$ such that $\overline{U}\subseteq U_1$.
    \begin{enumerate}
        \item Define the set $\mathcal{U}_{U,f}$ as the homotopy connected component containing $f$ of the set \[\mathcal{U}_U=\{g\in C^0(M,N)\mid g(M\setminus U)\subseteq N\setminus Y\}.\]
        Then $\mathcal{U}_{U,f}\subseteq C^0_S(M,N)$ is open with respect to Whitney's strong topology.
        \item If $g\in \mathcal{U}_{U,f}$ and $\widetilde{Z}=g^{-1}(Y)$, then there exist abelian groups $G_i$, for each $i\in \mathbb{N}$, such that
        \[\check{H}^i(\widetilde{Z})\simeq \check{H}^i(Z)\oplus G_i.\]
    \end{enumerate}
\end{theorem}
In the particular case where $f\in C^1(M,\R^k)$, $\delta(f)>0$ and the perturbation $g$ is also smooth and regular we can use singular cohomology and we have the following.
\begin{corollary}
 \label{semicontinuity}
    Let $M$ be a smooth manifold and let $f\in C^{1}(M,\R^k)$ be such that $f\pitchfork \{0\}$.
    For every $g \in C^1(M,\R^k)$ with $g\pitchfork \{0\}$ and
    \[\|f-g\|_{C^0(M,\R^k)}<\frac{\delta(f)}{2},\]
    we have $b(Z(f))\le b(Z(g))$.
\end{corollary}
\begin{proof}
    Let $U,U_1\subseteq M$ be open tubular neighborhoods of $Z(f)$ satisfying $\overline{U}\subseteq U_1$. We need to show that $g\in \mathcal{U}_{U,f}$, i.e. $Z(g)\subseteq U\subseteq U_1$,  and that there is a homotopy $f_t$ between $f$ and $g$ such that $Z(f_t)\subseteq U$.  
    As tubular neighborhoods we consider $$U\coloneqq f^{-1}(B_{\frac{\delta(f)}{2}}(0)), \,\;\;\;U_1\coloneqq f^{-1}(B_{\delta(f)}(0)).$$ The restriction $f: U \to B_{\frac{\delta(f)}{2}}(0) $ is  proper because $M$ is compact and is a submersion by the definition of $\delta(f)$. By Ehresmann's lemma $U\simeq B_{\frac{\delta(f)}{2}}(0)\times Z(f)$, proving that $U$ is a tubular neighborhood of $Z(f)$. The same is true for $U_1$.
    
    Now  let $f_t\coloneqq (1-t)f+tg$, then:
    \[\lVert f-f_t \rVert_{C^0} = \lVert t(f-g) \rVert_{C^0} <\frac{\delta(f)}{2}.\]
    Therefore if $x\in Z(f_t)$ then $\lVert f(x)\rVert<\frac{\delta(f)}{2}$, thus the claim follows by \Cref{general semicontinuity}.
\end{proof}

\subsection{Real algebraic geometry}
This section is primarily devoted to the proof of \Cref{theorem B}, but we will also recall some general facts from real algebraic geometry that will be used later.

First, we state the Nash–Tognoli theorem, which allows us, in the proof of \Cref{theorem A}, to replace $M$ with a diffeomorphic algebraic manifold $\widetildeto{A}{M}$.
Next, we recall the classical  Petrovsky--Oleinik--Thom--Milnor bound and we observe that it is insufficient to achieve the correct order in \Cref{theorem A}. Finally, we prove \Cref{theorem B}, which generalize the Petrovsky--Oleinik--Thom--Milnor bound to the case where the ambient manifold is not Euclidean but instead a regular algebraic manifold.

Let us denote by $P_{n,d}$ the space of polynomial $p\in \R[x_1,\dots,x_n]$ of degree $\le d$. 
\subsubsection{Nash--Tognoli Theorem}
\begin{definition}[Regular variety]
    Let $q_1,\dots,q_\ell\in P_{n,d}$ and let $Z(q_1,\dots,q_\ell)\subseteq \R^n$ be an algebraic variety  of dimension $m$. We say that the variety is regular if the rank of the Jacobian matrix $\left[\frac{\partial q_j(x)}{\partial x_i}\right]$ is equal to $n-m$ for any $x\in Z(q_1,\dots,q_\ell)$.
\end{definition}
We recall the Nash--Tognoli Theorem, see \cite{bochnakRealAlgebraicGeometry1998} Theorem 14.1.10. Using the Nash--Tognoli Theorem we can replace our Riemannian manifold $(M,g)$ by an algebraic manifold.  
\begin{theorem}[Nash--Tognoli Theorem] \label{Nash tognoli theorem}
    Let $M$ be a smooth compact manifold. Then there exists a regular algebraic variety $\widetildeto{A}{M}$ which is diffeomorphic to $M$.
\end{theorem}

\begin{remark}\label{rem:nash tognoli}
    Any non singular point $x$ of an algebraic variety $\widetildeto{A}{M}$ belongs to a unique irreducible component, see \cite[Proposition 3.3.10]{bochnakRealAlgebraicGeometry1998}. It follows that if $\widetildeto{A}{M}$ is connected and regular then it is also irreducible.
\end{remark}
\subsubsection{Oleinik--Petrovsky--Thom--Milnor bounds}
Next, we recall the following classical result, first due to Oleinik and Petrovsky and later refined by Thom and Milnor, which estimates the Betti numbers of an algebraic variety (not necessarily regular) in terms of the degree of the defining polynomial equations. For a proof, we refer the reader to \cite{milnorBettiNumbersReal1964}. 

\begin{theorem} \label{classical thom milnor}
    Let $p_1, \dots, p_k \in P_{n,d}$. Then:
    \[
        b(Z(p_1, \dots, p_k)) \leq d(2d - 1)^{n - 1}.
    \]
\end{theorem}
Even though this estimate is not sharp, the order $ O(d^n)$ is optimal.

If instead we consider zero sets restricted to an algebraic manifold $\widetildeto{A}{M}$ of dimension $m$, in analogy with \Cref{classical thom milnor}, we would expect a bound of order $ O(d^m) $. For example, many generalizations of \Cref{classical thom milnor} in the literature estimate the Betti numbers of semialgebraic sets restricted to an ambient semialgebraic set of fixed dimension (see, for example,  \cite{basuBoundingBettiNumbers1999, basuBettiNumbersSign2004, basuNumberCellsDefined1996}).
There, the Betti numbers of semialgebraic sets are bounded in terms of the degree $d$ of the defining equations and on the number $s$ of the equations. It is shown that, if we restrict to an ambient semialgebraic set of dimension $m$ living in $\R^n$, the bound, from $O(d^ns^n)$ can be improved to $O(d^ns^m)$.

In what follows, we will show that also the dependence on the degree can be improved to $ O(d^m)$.

\begin{remark}
A recent work by Basu and Parida \cite[Theorem 7]{basu2025boundsrealizationszerononzeropatterns} obtained a similar bound using different techniques. In their approach, however, they use the degree of the variety $M$ defined as the generic number of (complex) intersection points with a plane of complementary dimension. One can bound the degree of $M$ using the degree of the defining equations, but the resulting estimate is weaker: $b(M)= O(d_0^{m(n-m)}d^m)$.
\end{remark}

\subsubsection{Morse polynomials}
The proof of \Cref{theorem B} relies on  Morse theory -- {this is standard for arguments of this type.} More precisely, we will construct a polynomial which is a Morse function for $M$ and bound the number of its critical values by its degree.
We recall that a smooth map $f\colon M\to \R$ is a Morse function if $Df\colon M\to T^*M$ is transversal to the zero section $Z\subseteq T^*M$. 

We begin by some elementary facts about Morse polynomials.
\begin{prop} 
    Let $M\subseteq \R^n$ be a smooth compact manifold. Then, for any $d\ge 1$ the map
    \begin{align*}
        F\colon M\times P_{n,d}&\to T^*M\subset M\times (\R^n)^* \\
        (x,p)&\mapsto (x,D_xp|_{T_xM})
    \end{align*}
    is a submersion.
\end{prop}
\begin{proof}
    We prove that the map
    \begin{align} \label{morse map}
          \phi\colon M\times P_{n,d}&\to M\times (\R^n)^* \nonumber \\
        x,p&\mapsto (x,D_xp)
    \end{align}
    is a submersion. Since the first component $(x,p)\mapsto x$ is a submersion for each $p\in P_{n,d}$, it is sufficient to prove that for each $x\in M$ the map $p\mapsto D_xp$ is a submersion. Since the map is linear its differential is $q\mapsto D_xq$ for $q \in T_p P_{n,d}$. Let $v\in (\R^n)^*$, choose $q\in T_p(P_{n,d})=P_{n,d}$ such that $D_xq=v$, e.g. $q(x)=\langle v,x\rangle$. This is possible if $d\ge1$. We have that $$D_p(D_xp)(q)=D_xq=v,$$
    which proves that the map (\ref{morse map}) is submersion. We notice also that the  restriction operator 
    \begin{align*}
        \Pi\colon M\times (\R^n)^*&\to T^*M \\
       x, v&\mapsto x,v|_{T_xM}
    \end{align*}
    is a submersion, because it is the identity on the first coordinate and a surjective linear function on the second coordinate. Therefore we obtain that the map $F=\Pi\circ \phi$ is a submersion.
\end{proof}
As a corollary we have the following.
\begin{corollary}    \label{morse}
    Let $M\subseteq \R^n$ be a smooth compact manifold, and  let $W\subseteq M$ be a stratified set of codimension at least $1$. Then, for $d\ge 1$, there exists a residual subset $A_W \subseteq P_{n,d}$ such that for any $r\in A_W$, the restriction $r|_M$ is a Morse function for $M$ with $\crit{r|_M}\subseteq M\setminus W$.
\end{corollary}
\begin{proof}
    Let $Z\subseteq T^*M$ be the zero section, and let $Z_W\coloneqq \{(x,0)\mid x\in W\}\subseteq T^*M$. By the Parametric transversality Theorem, see \cite{leeIntroductionSmoothManifolds2012} Theorem 6.35,  there exists a residual set of polynomials $r\in P_{n,d}$ such that $x\mapsto (x,D_xr|_{T_xM})$ is transversal to both $Z$ and $Z_W$. If  $x\mapsto(x,D_xr)$ is transversal to $Z$ then, by definition, $r$ is Morse. Now, consider a smooth stratum $W_i$ of $Z_W$. This is a smooth submanifold of $T^*M$ with codimension at least $n+1$. Consider      \begin{align*}
        g\colon M &\to T^*M \\
        x&\mapsto (x,D_xr).
    \end{align*}
    Transversality of $g$ to a submanifold $Y$  means that for each $x\in M$ such that $g(x)\in Y$, the following holds $$\im D_xg \oplus T_{g(x)}Y= T_{g(x)}T^*M.$$
    Since $\dim \im D_xg+ \dim T_{g(x)}Y\le n+n-1<2n$,  it  must be the case that $g(x)\not \in Y$ for each stratum $Y$ of $Z_W$.  Therefore $\crit{r}|_M\subseteq M\setminus W$ as desired.
\end{proof}
\subsubsection{Proof of \Cref{theorem B}}
{Everything is now ready for the proof of \cref{theorem B}. Let us recall the statement of the theorem.}

\begingroup
\def\thethm{\ref{theorem B}}
\begin{thm} 
    Let $M=Z(q_1,\dots, q_\ell)$ be a regular, irreducible compact manifold of dimension $m$, with $\deg q_i\le d_0$. For any $d\ge 1$ consider $p_1,\dots, p_k \in P_{n,d}$,  then:
    \[b(Z(p_1,\dots,p_k)\cap M)\le d_0^{n-m}((n-m)(d_0-1) + 2d-1)^{m}=O(d^m).\]
\end{thm}
\addtocounter{thm}{-1}
\endgroup

\begin{proof}
    First notice that
    $Z(p_1,\dots, p_k)=\{\sum_{i=1}^k p_i^2\le0\}$. By semialgebraic triviality \cite[Proposition 9.4.4]{bochnakRealAlgebraicGeometry1998} for every $\epsilon >0 $ small enough, the inclusion $$\bigg\{\sum_{i=1}^k p_i^2\le 0\bigg\}\xhookrightarrow{\sim}\bigg\{\sum_{i=1}^k p_i^2\le \epsilon\bigg\}$$
    is a homotopy equivalence. Assume now that $\epsilon$ is a regular value of $\sum_1^k p_i^2$ or, equivalently, that $0$ is a regular value of $$f \coloneqq \sum_{i=1}^k p_i^2-\epsilon.$$  
    Let $q=(q_1,\dots,q_\ell)$ and for $x_0\in M$ consider  $\rho_{x_0}(x)$ a $(n-m)$--minor of the Jacobian matrix $Jq(x)$. Since $M$ is regular we can assume $\rho_{x_0}(x)$ to be non vanishing at $x_0$. Since $M$ is irreducible $Z(\rho_{x_0})\cap M$ has codimension at least $1$. We can apply  \Cref{morse} 
    to find $r\in P_{n,2d}$ Morse function for $M$, such that \mbox{$\crit{r}|_{M}\subseteq M\setminus Z(\rho_{x_0}) $} and: \[\|r-f\|_{C^1(M,\R)}<\delta(f).\]
   By \Cref{stability lemma}, the sublevel sets $\{f\le 0\}\cap M$ and $\{r\le 0\} \cap M$ are diffeomorphic, therefore, by the Morse inequalities \cite[Theorem 5.2]{milnorMorseTheoryAM511963}, we obtain:  $$b(Z(p_1,\dots,p_k))=b\left(\{f\le 0\}\right)\le \#\crit{r|_M}.$$    
    To simplify the notation assume that the minor $\rho_{x_0}$ is computed using the first $n-m$ rows and the first $n-m$ columns of the Jacobian matrix, i.e.:$$\rho_{x_0}(x)=\det\left[\frac{\partial q_j(x)}{\partial x_i}\right]_{\substack{i=1,\dots,n-m\\j=1,\dots,n-m}}$$ 
    If $x$ is a critical point of $r|_M$ then it satisfies the system:
    \begin{equation}\label{prop5eq1}
    \begin{cases}
      q_1(x)=\dots = q_{n-m}(x)=0 \\
      \nabla|_{T_xM} r(x)=0,
    \end{cases}
    \end{equation}
    where $\nabla|_{T_xM} r(x)=0$ denotes the gradient of $r(x)$ along a basis of $T_xM$.
    Denoting by $\nabla_{n-m}\coloneqq(\partial_1,\dots,\partial_{n-m})$, a basis of $T_xM$  for  $x\in M\setminus Z(\rho_{x_0})$ is described by:
    \[X_k(x)\coloneqq \det \left(\begin{array}{ ccc | c }
    &&\\
    \nabla_{n-m} q_1(x) &  \cdots  &\nabla_{n-m} q_{n-m}(x) & \nabla_{n-m}   \\
    &&\\
    \hline
    \partial_k q_1(x)& \cdots  &\partial_k q_{n-m}(x)& \partial_k 
    \end{array}\right),  \text{ for } k=n-m+1,\dots n\]  
    Indeed, if $x\in \R^n\setminus Z(\rho_{x_0})$, the  $\nabla q_i(x)$ form a basis of $N_xM$ for $i=1,\dots, n-m$. Since $X_k(x)$ are also linearly independent and  $\langle X_k(x), \nabla q_i(x)\rangle=0$, it follows that the $X_k(x)$ form a basis of $T_xM$.
    The system (\ref{prop5eq1}) can be rewritten as:
    \begin{equation}\label{prop5eq2}
    \begin{cases}
      q_k(x)=0 &  \text{for } k=1,\dots n-m,\\
      X_k (r)(x)=0   & \text{for } k=n-m+1,\dots n.
    \end{cases}
    \end{equation}
    To check that the system is nonsingular we compute its Jacobian using the basis of $T_x\R^n=N_xM\oplus T_xM$ given by a basis of $N_xM$ and by the $X_h(x)$'s:
\begin{equation*}       
       \left(\begin{array}{ c | c }
    \nabla_{N_xM} q_i(x) & X_h q_i(x)\\
    \hline
    \nabla_{N_xM} X_h (r)(x)& X_h X_k r(x)
    \end{array}\right)
\end{equation*}
If $x\in M \setminus Z(\rho_{x_0})$, then $\nabla q_i$ are a basis of $N_xM$, so $\det(\nabla_{N_xM} q_i)\not=0$. Since $X_h\in T_xM$, then $X_h q_i(x) =0$, because $q_i(x)=0$ for all  $x\in M$. We also have that $\det(X_h X_k r)\not=0$ because $r$ is Morse for $M$.
Applying Bezout Theorem, see \cite[Lemma 11.5.1]{bochnakRealAlgebraicGeometry1998},  the number of complex projective non singular zeroes of the polynomial systems \eqref{prop5eq2} is bounded by the product of the degrees of the polynomials, {getting, as claimed}:  $$\# \text{crit } r|_M\le d_0^{n-m} ((n-m)(d_0-1)+(2d-1))^m.$$
\end{proof}
\subsection{Finite families of maps}
In this section, we introduce the basic notions needed to define semialgebraic--type sets, as in \cref{sat}. In particular, we recall the definition of a Whitney stratified set, introduce the notion of a finite family of maps and its corresponding condition number. We conclude the section with a general version of the Thom isotopy lemma and an inequality concerning the Betti numbers of finite intersections and finite unions of closed sets, which follows from the Mayer--Vietoris sequence.
\subsubsection{Whitney stratifications}
Recall from \cite{matherNotesTopologicalStability2012} the notion of Whitney's condition (b).
Consider $X,Y$ smooth submanifold of $M$, not necessarily closed nor connected. We say that the pair $(X,Y)$ satisfies Whitney's condition (b) at $y$ if the following holds.
For any sequence $\{x_n\}\subseteq X$, $\{y_n\}\subseteq Y$ both converging to $y$ suppose that $T_{x_n}X$ converges to some $\tau \subset T_yM$ and that $x_n\not= y_n$ for any $n$, and the secants $\overline{x_ny_n}$ (in local coordinates) converge to some line $\ell \subseteq T_yM$, then $\ell \subseteq \tau$. 
\begin{definition}
    Let $M$ be a smooth manifold and let $S\subseteq M$. A \textit{stratification} $\mathcal{S}$ of $S$ is a cover of $S$ by finite pairwise disjoint smooth submanifolds of $M$ contained in $S$. The stratification satisfies the \textit{frontier condition} if for each stratum $X\in \mathcal{S}$ its frontier $\overline{X}\setminus X \cap S$ is a union of strata.
    We say $\mathcal{S}$ is a Whitney stratification if it satisfies the frontier condition and $(X,Y)$ satisfies condition (b) for any pair $(X,Y)$ of strata of $\mathcal{S}$. In this case, $S$ is called Whitney stratified, and its dimension is given by the {maximal} dimension among its strata.
\end{definition}
We remark that any subset of a Whitney stratification $\mathcal{T}\subseteq \mathcal{S}$ is itself a Whitney stratification.
Transversal intersections of Whitney stratified spaces remains Whitney stratified, with the new strata given by all  possible intersections of the original strata. Given $S_1$ and $S_1$ with stratifications $\S_1$ and $\S_2$, then we denote by $\S_1\times \S_2$ the product stratification of $S_1\times S_2$, whose strata are given by all the possible products of strata of $\S_1$ and $S_2$.

\begin{definition}
If $f\colon M\to N$ is a smooth map,  we say that 
 $f$ is transversal to the Whitney stratified space $S$ and we denote this as $f\pitchfork \mathcal{S}$ if  $f$ is transversal to each smooth stratum of $S$. In this case $f^{-1}(S)$ is stratified by the pullback stratification $$f^{-1}(\mathcal{S})\coloneqq \{f^{-1}(S_i)\mid S_i \in \mathcal{S}\}.$$
Since for each stratum $S_i$, the codimension of $S_i$ coincides with the codimension of its preimage $f^{-1}(S_i)$,  we have that $$\codim(S)=\codim f^{-1}(S).$$
\end{definition}
We now introduce the main stratification that we are going to use on $\R^s$.
\begin{definition}\label{def:orthant strat}
We denote by $\mathcal{S}_{I}$, for $I\subseteq \{1,\dots,s\}$, the stratification of $\R^s$ given by the following  strata  $$\bigg\{\sign(x_i)=\sigma\;\bigg|\; \sigma\in\{-1,0,1\}, i\in I\bigg\}.$$ 
If $I=\{1,\dots,s\}$ we call it the orthant stratification of $\R^s$ and we denote it just by $\mathcal{S}$.
\end{definition}
\begin{figure}[H]
\begin{mdframed}[userdefinedwidth=5cm,align=center]\centering
\scalebox{0.75}{\includegraphics[]{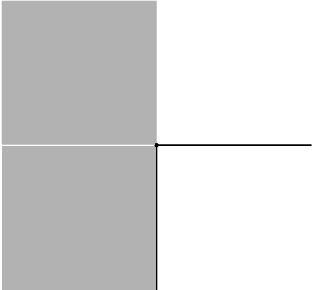}}
\end{mdframed}
\caption{Example of strata of $(\R^2,\mathcal{S})$}
\end{figure}
\subsubsection{Condition number of a finite family}The following definition is inspired by \cite{burgisserComputingHomologySemialgebraic2021}, where a similar notion is introduced for a family of polynomials (see \cite[Definition 3.7]{burgisserComputingHomologySemialgebraic2021}).
\begin{definition}\label{smooth family definition}
    A smooth family is a finite subset $\F\coloneqq\{f_1,\dots,f_s\}\subseteq C^1(M,\R)$.  
    For any $J\subseteq\{1,\dots , s\}$, we denote by $f_J\coloneqq (f_j)_{j\in J}\colon M \to \R^{|J|}$.
    We will write $f\colon M \to \R^s$ in place of $f_{\{1,\dots, s\}}$. 
    A family $\F$ is said regular if   for any \mbox{$J\subseteq \{1,\dots,s\} $}, $0$ is a regular value for $f_J$.
\end{definition}
\begin{remark} \label{equivalent description of regularity}
     Since $f_J\pitchfork 0$ if and only if $f\pitchfork \{x_j=0\mid j\in J\}$, then it follows that a family $\F$ is regular if and only if $f: M \to \R^s$ is transversal \mbox{to $(\R^s,\mathcal{S})$}. In this case we can stratify $M$ by the pullback stratification $f^{-1}(\mathcal{S})$.
\end{remark}
In analogy with \Cref{def:discriminant}, we have the following definition of discriminant for finite families. By \cref{equivalent description of regularity} we identify families $\F$ of $s$ elements by functions $f\in C^1(M,\R^s)$.
\begin{definition}\label{def:discriminant families}[Discriminant for families]
    We define the discriminant $\Delta_s$ for family of functions $f_i\colon M\to \R$, where $i=1,\dots,s$, as $$\Delta_s\coloneqq \left\{f\in C^1(M,\R^s)\,\bigg|\, f\not \pitchfork \mathcal{S}\right\}\subseteq C^1(M,\R^s).$$
    Also the distance from the discriminant $\delta(\F)$ and the condition number $k(\F)$ are defined by:
    $$\delta(\F)\coloneqq \dist_{C^1(M,\R^k)}(f,\Delta_s), \quad\quad \quad k(\F)\coloneqq \frac{\|f\|_{C^1}}{\delta(\F)}.$$
\end{definition}
Since transversality to a Whitney stratified set is an open condition (see \cite[Proposition 1.3.4]{goreskyStratifiedMorseTheory1988}), the set $\Delta_s$ is closed. It follows that a family $\F$ is regular if and only if $\delta(\F)>0$.

Again in analogy with the case of a map, we introduce $\Sigma_{s}\subseteq J^1(M,\R^s)$ consisting of the singular jets. The subscript $s$ denotes the cardinality of the families that we are considering.
\begin{definition}
    Let $\Sigma_{s,x}\subseteq J^1_x(M,\R^s)$ be the set
    \[\Sigma_{s,x}\coloneqq \left\{j^1f(x)\,\bigg|\, f\not \pitchfork \mathcal\mathcal{S} \text{ at } x\right\},\]
    We define $\Sigma_{s}\coloneqq \bigcup_{x\in M}\Sigma_{s,x}$. 
\end{definition}
The following lemma establishes a correspondence between the distance from the discriminant of the family $\F$, as defined in \cref{def:discriminant families}, and the distances from the discriminants of the functions $f_I$ as defined in \cref{def:discriminant}.
\begin{lemma}\label{proposition distance discriminant for families}
    We have that 
    \[\delta(\F)=\min_{\substack{J\subseteq \{1,\dots,s\}\\ |J|\le m+1}}\delta(f_J)\]
\end{lemma}
\begin{proof}
    Since $\Delta_s$ is closed there exists $g\in \Delta_s$ such that $$\dist_{C^1(M,\R^k)}(f,\Delta_s)=\|f-g\|_{C^1}.$$ Consider $J\subseteq\{1,\dots,s\}$ such that $g_J\not \pitchfork 0$. Then $$\delta(\F)=\|f-g\|_{C^1}\ge\|f_J-g_J\|_{C^1}\ge \delta(f_J).$$
    For the reverse inequality, consider $h\in C^1(M,\R^{|J|})$, for some $J\subseteq\{1,\dots,s\}$, such that $\|f_J-h\|_{C^1}=\min_{J\subseteq \{1,\dots,s\}}\delta(f_J)$. Now consider $g\colon M\to \R^s$ such that $g_J=h$ and $g_i(x)=f_i(x)$ for $i\not \in J.$ Since $h\not \pitchfork 0$, we have that $g\in \Delta_s$ and therefore
    \begin{equation}\label{equationlemma2.28}\delta(\F)\ge \|f-g\|_{C^1}=\|f_J-g_J\|=\min_{J\subseteq \{1,\dots,s\}}\delta(f_J).\end{equation}
    If $|J|>m+1$, consider $J_0\subseteq J$ such that $|J_0|=m+1.$ Since $|J|>|J_0|>m$, then both $D_x f_J$ and $D_x f_{J_0}$ cannot be surjective, therefore we have
    \begin{align*}\delta(f_J)=\min_{x\in M} (\|f_J\|+\sigma_{|J|}(D_xf_J))=&\min_{x\in M} \|f_J\|\\\ge& \min_{x\in M}\|f_{J_0}\|
    \\ =&\delta(f_{J_0}).
    \end{align*}
    and this implies that the minimum is realized by a function $f_{J_0}$ such that $|J_0|\le m+1.$
\end{proof}
Similarly to the case of functions, we can compute the distance from the discriminant $\delta(\F)$ by using $\Sigma_s$ instead of $\Delta_s$:
\begin{prop}
    The following equality holds \[\delta(\F)= \inf_{x\in M}\dist(j^1f(x),\Sigma_{s,x}).\]
\end{prop}
\begin{proof}
    For $J\subseteq\{1,\dots,s\}$, denote by $\Sigma_{J,x}\subseteq J^1_x(M,\R^s)$ and by $\hat{\Sigma}_{J,x}\subseteq J^1_x(M,\R^{|J|})$ the sets 
    \[\Sigma_{J,x}=\left\{j^1f(x)\,\bigg|\, f \not \pitchfork \{x_j=0\mid j\in J\}\text{ at }x\right\} , \;\hat{\Sigma}_{J,x}=\left\{j^1f_J(x)\,\bigg|\, f_J \not \pitchfork \{0\}\text{ at }x\right\}. \]
    By \cref{proposition distance discriminant for families} and \cref{equality distance discriminant} we have
    \begin{align}
        \delta(\F)=&\min_{J\subseteq \{1,\dots,s\}}\delta(f_J)\nonumber\\
        =&\min_{J\subset\{1,\dots,s\}}\left(\min_{x\in M}\dist(j^1f_J(x),\hat{\Sigma}_{J,x})\right)\label{eqprop:2.29}.
    \end{align}
    We claim that \begin{equation}\label{eqprop:2.29 2}\dist(j^1f_J(x),\hat{\Sigma}_{J,x})=\dist(j^1f(x),\Sigma_{J,x}).\end{equation}
    Clearly $$\dist(j^1f_J(x),\hat{\Sigma}_{J,x})\le\dist(j^1f(x),\Sigma_{J,x})$$ because if $j^1g_J(x)\in \hat{\Sigma}_{J,x}$ then also $j^1g(x)\in \Sigma_{J,x}$ and $\|j^1f_J-j^1g_J\|\le \|j^1f-j^1g\|$.
    Conversely let $g_J\in C^1(M,\R^{|J|})$ such that $j^1g_J\in \hat{\Sigma}_{J,x}$ realizes the distance $\dist(j^1f_J(x),\hat{\Sigma}_{J,x})$, then define $g\in C^1(M,\R^s)$ by $g_i=g_j$ for $i\in J$ and $g_i=f_i$ otherwise. We have that \[\dist(j^1f(x),\Sigma_{J,x})\le \|j^1f-j^1g\|=\|j^1f_J-j^1g_J\|=\dist(j^1f_J(x),\hat{\Sigma}_{J,x}).\]
    Substituting \eqref{eqprop:2.29 2} in \eqref{eqprop:2.29} we obtain
    \[\delta(\F)=\min_{J\subset\{1,\dots,s\}}\left(\min_{x\in M}\dist(j^1f(x),{\Sigma}_{J,x})\right)\]
    Since $j^1g\in \Sigma_{s,x}$ if and only if $j^1g\in \Sigma_{J,x}$ for some $J\subseteq\{1,\dots,s\}$, we obtain the claim.
\end{proof}
\subsubsection{Thom isotopy Lemma}
In the sequel, we will often perturb the functions of a family $\F$. The main tool we will use to deal with these perturbations is Thom's First Isotopy Lemma. One of its consequences is that the stratified sets described on a regular family $\F$ are stable under small perturbations of the family.

The following result can be found in \cite[Chapter 1.5]{goreskyStratifiedMorseTheory1988}. 
\begin{theorem}[Thom Isotopy Lemma] \label{Thom isotopy general}
    Let $M$ be a smooth manifold and let $\mathcal{M}$ be a Whitney stratification of $M$. Consider $f\colon M\to \R^s$ and assume that $f$ is a proper and that it is a  submersion when restricted to each stratum $M_i\in \mathcal{M}$. For any $y \in \R^s$, consider $f^{-1}(y)$, which is a smooth submanifold stratified by $\mathcal{M}\cap f^{-1}(y)$. Then there is a stratum preserving homeomorphism $\phi$ which is smooth on each stratum and makes the following diagram commute:
   \[\begin{tikzcd}
	{(M,\mathcal{M})} & {\R^s\times\left(f^{-1}(y),f^{-1}(y)\cap \mathcal{M}\right) } \\
	\R^s
	\arrow["\phi", from=1-1, to=1-2]
	\arrow["f"', from=1-1, to=2-1]
	\arrow["{\pi_1}", from=1-2, to=2-1]
\end{tikzcd}\]
   
\end{theorem}
A consequence of the Isotopy Lemma is given in the following proposition. It describes the stratification $f^{-1}(\mathcal{S})$ given by a regular family $\F$ when  restricted to $ f_I^{-1}(B_\delta)$, where $B_\delta=\{y\mid\|y\|<\delta \} \subseteq \R^{|I|}$.
\begin{prop} \label{prop: product stratification}
    Let $\F=\{f_1,\dots,f_s\}$ be a smooth family with $\delta(\F)>\delta$. Then, for each $I\subseteq\{1,\dots,s\}$, there exists a stratum preserving  homeomorphism $$\phi_I\colon f_I^{-1}(B_\delta)\to B_\delta\times f_I^{-1}(0),$$
    between the stratification $f^{-1}(\mathcal{S})=f^{-1}(\S_I\times \S_{\{1,\dots,s\}\setminus I}$ of $f_I^{-1}(B_\delta)$  and the stratification $\mathcal{S}_I\times f^{-1}(\mathcal{S}_{\{1,\dots,s\}\setminus I})$ of $B_\delta\times f_I^{-1}(0)$, where $\S_I$ denotes the orthant stratification of $B_\delta\subseteq \R^{|I|}$. Moreover, this homeomorphism satisfies $$\pi_1\circ\phi_I=f_I.$$ 
\end{prop}
\begin{proof}
If $f_I^{-1}(B_\delta)=\emptyset$ there is nothing to prove. Otherwise, consider $f_I^{-1}(B_\delta)$ stratified by $f^{-1}(\mathcal{S}_{\{1,\dots,s \}\setminus I})$. We claim that the map 
\[f_I|_{f_I^{-1}(B_\delta)}\colon f_I^{-1}(B_\delta)\to B_\delta,\]
is proper and  a submersion on stratum. Properness follows from the compactness of $M$,since  $f_I\colon M \to \R^{|I|}$ is continuous. Let $C\subseteq f_I^{-1}(B_\delta)$ be a non-empty stratum of $f^{-1}(\mathcal{S}_{\{1,\dots,s \}\setminus I})$, and let $x\in C$. Define $$J\coloneqq\bigg\{j\in \{1,\dots, s\}\setminus I\;\bigg |\; f_j(x)=0 \bigg\}.$$
Note that $|J|$ equals the codimension of $C$ in $f_I^{-1}(B_\delta)$. We need to show that  the differential
\[D_x f_I|_{T_xC}=D_xf_I|_{\ker D_xf_J}\] is surjective. Since $\delta(\F)>\delta,$  by \cref{proposition distance discriminant for families}, the values $\delta(f_I)$, $\delta(f_J)$ and $\delta(f_{I\cup J})$ are all greater than $\delta.$ Moreover, since $\|f_I(x)\|,$ $\|f_J(x)\|$ and $\|f_{I\cup J}(x)\|$ are each less than $\delta$, it follows from \cref{lem: distance discriminant = distance sigma} and \cref{equality distance discriminant} that $D_xf_I$, $D_x f_J$ and $D_x f_{I\cup J}$ are all surjective. Using  the rank--nullity theorem, we have \[\dim\ker( D_xf_I|_{\ker D_xf_J})+ \dim \im(D_xf_I|_{\ker D_xf_J})=\dim \ker D_xf_J. \]
Now $\dim\ker( D_xf_I|_{\ker D_xf_J})=\dim \ker D_x f_{I\cup J}=m-|I|-|J|$ and $\dim \ker D_xf_J=m-|J|$, hence $$\dim \im(D_xf_I|_{\ker D_xf_J})=|I|,$$
proving surjectivity.

By \cref{Thom isotopy general}, we obtain that there exists a stratum preserving homeomorphism \begin{equation}\label{eq: thom iso proposition}\phi_I\colon  f_I^{-1}(B_\delta)\to  B_\delta\times f_I^{-1}(0)\end{equation} where $f_I^{-1}(0)$ and $f_I^{-1}(B_\delta)$ are stratified by $f^{-1}(\mathcal{S}_{\{1,\dots,s \}\setminus I})$. 

Now we claim that, if we additionally stratify $B_\delta$ using the orthant stratification $\S_I$, then $\phi_I$ becomes stratum preserving for the finer stratification $f^{-1}(\mathcal{S})$ of  $f_I^{-1}(B_\delta)$.
Observe that the orthant stratification $\mathcal{S}$ of $\R^s$ is the transverse intersection of $\mathcal{S}_I$ and $\mathcal{S}_{\{1,\dots,s\}\setminus I}$, i.e. \begin{equation}\label{eq:product stratification}\mathcal{S}=\mathcal{S}_I\cap\mathcal{S}_{\{1,\dots,s\}\setminus I}.\end{equation} 
Since $f$ is transverse to $\mathcal{S}$, taking  preimages in \eqref{eq:product stratification}, we obtain 
\[f^{-1}(\mathcal{S})=f^{-1}(\mathcal{S}_I)\cap f^{-1}(\mathcal{S}_{\{1,\dots,s\}\setminus I}).\]
It remains to show that $\phi_I$ in \eqref{eq: thom iso proposition} is stratum preserving with respect to the  stratification $f^{-1}(\mathcal{S}_I)$, i.e. that it maps strata of  $(f^{-1}_I(B_\delta),f^{-1}(\mathcal{S}_I))$ to strata of $(B_\delta\times f^{-1}(0),\S_I\times f^{-1}(0))$. 
Let $f^{-1}(T)$ be a stratum of $f^{-1}(\mathcal{S}_I)$, then by \cref{def:orthant strat}, $$T=\bigg\{\sign x_i =\sigma \;\bigg|\;\sigma\in \{-1,0,1\}, i\in I\bigg\}.$$ Thus the map $$f_I\colon (f^{-1}(B_\delta),f^{-1}(\S_I))\to( B_\delta,\S_I)$$ is stratum preserving. Since $\pi_1\circ\phi_I=f_I$, it follows that
\[\phi_I\colon (f^{-1}(B_\delta),f^{-1}(\S_I))\to( B_\delta,\S_I)\times f^{-1}(0) \]
is also stratum preserving. 
\end{proof}
Sometimes it is important to perturb each function in a family $\F$ by a constant $\delta>0$. The next remark addresses this case.
\begin{remark} \label{rem: product stratification}
    If $\delta<\sqrt{m+1}\cdot\delta(\F)$, then $f_I^{-1}\big((-\delta,\delta)^{|I|}\big)\subseteq f_I^{-1}(B_{\delta(\F)})$ for each $I\subseteq\{1,\dots,s\}$. Indeed, if $|I|\le m+1$, then $(-\delta,\delta)^{|I|}\subseteq B_{\delta(\F)}\subseteq \R^{|I|}$, while if $|I|>m+1$ consider $I_0\subset I$, $|I_0|=m+1$, then
    \[f_I^{-1}\big((-\delta,\delta)^{|I|}\big)\subseteq f_{I_0}^{-1}\big((-\delta,\delta)^{|I_0|}\big)\subseteq f_{I_0}^{-1}(B_{\delta(\F)}).\]
    Since $f_{I_0}$ cannot be a submersion because (the dimension of the codomain is bigger than the dimension of the domain), then both $f_{I_0}^{-1}(B_{\delta(\F)}) $ and $f_{I}^{-1}(B_{\delta(\F)})$ are empty.
    This implies that if we consider any stratifications $\S'_{\{1,\dots,s\}\setminus I}$ of $\R^{s-|I|}$ given by the transversal intersections of the hypersurfaces \[\bigg\{x_j=\lambda_{j,k}\;\bigg|\;j\in \{1,\dots,s\}\setminus I, \bigg\} ,\] for some constants $|\lambda_{j,k}|\le \delta$, then $f_I\colon f_I^{-1}((-\delta,\delta)^{|I|})\to (-\delta,\delta)^{|I|}  $
    is a proper submersion restricted to each stratum of $f^{-1}(\S'_{\{1,\dots,s\}\setminus I})$.
    Therefore, by the proof of \cref{prop: product stratification}, we obtain that there exists a stratum preserving homeomorphism 
    \[\phi_I\colon f_I^{-1}\big((-\delta,\delta)^{|I|}\big)\to (-\delta,\delta)^{|I|}\times f_I^{-1}(0),\]
    between the stratifications $f^{-1}( \S_I\times \S'_{\{1,\dots,s\}\setminus I})$ and $\S_I\times f^{-1}(\S'_{\{1,\dots,s\}\setminus I})$.
\end{remark}
The following theorem is also sometimes  referred to as the Thom isotopy Lemma, even though the two results are different.  The theorem states that, given a family of maps all transversal to a stratified set $\mathcal{N}$, the preimages are all stratified sets that are diffeomorphic, stratum by stratum. Notice that while the previous theorem guarantees the existence of a stratumwise homeomorphism, this theorem ensures a stratumwise diffeomorphism.  
It can be found in \cite[Theorem 2.D.2]{thomEnsemblesMorphismesStratifies1969}. 
\begin{theorem}
    
 \label{thom isotopy particular} 
    Let $M,N$ be smooth manifolds and consider $\mathcal{N}$ a Whitney stratification of $N$. Let  $f_t\colon M\to N$ and  assume that $f_t\pitchfork \mathcal{N}$ for each $t\in[0,1]$. Then:
    \[(M,f_0^{-1}(\mathcal{N}))\simeq (M,f_1^{-1}(\mathcal{N})).\]
    in particular if $W\subseteq N$ is a smooth stratum in $\mathcal{N}$, then $f_0^{-1}(W)$ and $f_0^{-1}(W)$ are ambient isotopic submanifold of $M$.
\end{theorem}

\subsubsection{Mayer Vietoris inequalities}
\label{section mayer vietoris} The strategy to prove \Cref{theorem C}  is to reduce to the simple case of a single function. To achieve this, it is necessary to compute the Betti numbers of unions and intersections of these simple sets. From the Mayer–-Vietoris exact sequence, see \cite{spanierAlgebraicTopology1981} chapter 4 section 6, we have that for any two closed set $C_1,C_2\subseteq M$ we have
\begin{equation}\label{MV union}
        b_i(C_1\cup C_2)\le b_i(C_1)+b_i(C_2)+b_{i-1}(C_1\cap C_2),
    \end{equation}
and
\begin{equation*}\label{MV intersection}
    b_i(C_1\cap C_2)\le b_i(C_1)+b_i(C_2)+b_{i+1}(C_1\cup C_2).
    \end{equation*}
and also 
\begin{equation}\label{MV parts}
    b_i(C_1)+b_i(C_2)\le b_i(C_1\cap C_2)+b_{i}(C_1\cup C_2).
\end{equation}
More generally we can prove the following proposition.
\begin{prop} \label{Mayer Vietoris}
    Let $C_j\subseteq M$, $j =1,\dots, s$ be closed sets. Then
    \begin{enumerate}
        \item \label{case 1} for each $0\le i \le m$
        \[
        b_i\Big(\bigcup_{j=1}^sC_j\Big)\le \sum_{\ell=1}^{i+1}\sum_{\substack{L\subseteq \{1,\dots,s\}\\ |L|=\ell}}b_{i-\ell+1}\Big(\bigcap_{j\in L}C_j\Big)
        ;\]
        \item \label{case 2} for each $0\le i \le m$, if $b_m(\partial C_j)=0$ for every $1\le j\le s$ then
        \[
        b_i\Big(\bigcap_{j=1}^sC_j\Big)\le \sum_{\ell=1}^{m-i}\sum_{\substack{L\subseteq \{1,\dots,s\}\\ |L|=\ell}}b_{i+\ell-1}\Big(\bigcup_{j\in L}C_j\Big)+\binom{s}{m-i}b_m(M)
        .\]
    \end{enumerate}
\end{prop}
\begin{proof}
    Let us prove \eqref{case 1}. If $s=1$ the inequality is trivially true.
    Assume by inductive hypothesis that it holds for $s-1$, then by \eqref{MV union}:
    
    \begin{align*}
    b_i\Big(\bigcup_{j=1}^s C_j\Big)\le b_i\Big(\bigcup_{j=1}^{s-1}C_j\Big)+b_i(C_s)+b_{i-1}\Big(\bigcup_{j=1}^{s-1}(C_j\cap C_s)\Big).
    \end{align*}
    Using the inductive hypothesis for $b_i\big(\bigcup_{j=1}^{s-1}C_j\big)$ and $b_{i-1}\big(\bigcup_{j=1}^{s-1}(C_j\cap C_s)\big)$, we obtain:
    \[b_i\Big(\bigcup_{j=1}^{s} C_j\Big)\le \sum_{\ell=1}^{i+1}\sum_{\substack{L\subseteq \{1,\dots,s-1\}\\ |L|=\ell}}b_{i-\ell+1}\Big(\bigcap_{j\in L}C_j\Big)+b_i(C_s)+\sum_{\ell=1}^{i}\sum_{\substack{L\subseteq \{1,\dots,s-1\}\\ |L|=\ell}}b_{i-\ell}\Big(\bigcap_{j\in L}C_j\cap C_s\Big)\]
    The last expression is equal to  $\sum_{\ell=1}^{i+1}\sum_{\substack{L\subseteq \{1,\dots,s\}\\ |L|=\ell}}b_{i-\ell+1}\big(\bigcap_{j\in L}C_j\big)$ proving the statement. 

    For \eqref{case 2}, consider $s=1$. If $0\le i\le m-1$, the inequality clearly holds. If $i=m$ then $b_m(C_1)\le b_m(M)$, indeed, let $Z=\overline{(M\setminus C_1)}$, by \eqref{MV parts}, 
    \[b_m(C_1)+b_m(Z)\le b_m(M)+b_m(\partial C_1).\]
    Since by hypothesis $b_m(\partial C_1)=0$, then the base case is proven.
    Assume now that case \eqref{case 2} holds for $s-1$. Then by \eqref{MV intersection}
    \begin{equation}\label{eq:prop MV} b_i\Big(\bigcap_{j=1}^s C_j\Big)\le b_i\Big(\bigcap_{j=1}^{s-1}C_j\Big)+b_i(C_s)+b_{i+1}\Big(\bigcap_{j=1}^{s-1}(C_j\cup C_s)\Big).\end{equation}
    Using the inductive hypothesis we obtain
    \begin{align*}
        b_i\Big(\bigcap_{j=1}^{s-1} C_j\Big)\le &\sum_{\ell=1}^{m-i}\sum_{\substack{L\subseteq \{1,\dots,s-1\}\\ |L|=\ell}}b_{i+\ell-1}\Big(\bigcup_{j\in L}C_j\Big)+\binom{s-1}{m-i}b_m(M),\end{align*}
    and 
    \[b_{i+1}\Big(\bigcap_{j=1}^{s-1} (C_j\cup C_s)\Big) \le \sum_{\ell=1}^{m-i-1}\sum_{\substack{L\subseteq \{1,\dots,s-1\}\\ |L|=\ell}}b_{i+\ell}\Big(\bigcup_{j\in L}C_j\cup C_s\Big)+\binom{s-1}{m-i-1}b_m(M).\]
    Substituting in \eqref{eq:prop MV} we get
    \[b_i\Big(\bigcap_{j=1}^s C_j\Big)\le \sum_{\ell=1}^{m-i}\sum_{\substack{L\subseteq \{1,\dots,s\}\\ |L|=\ell}}b_{i+\ell-1}\Big(\bigcup_{j\in L}C_j\Big)+\binom{s}{m-i}b_m(M).\]
\end{proof}


\section{Betti numbers of regular zero sets}
\subsection{Proof of Theorem A}
Before its proof, we restate the theorem.
\begingroup 
\def\thethm{\ref{theorem A}}
\begin{thm} 
    Let $(M,g)$ be a compact Riemannian manifold of dimension $m$. There exists a constant $c_1=c_1(M,g)>0$ such that for any \mbox{$f\in C^1(M,\R^k)$}, we have: 
    \[b(Z(f))\le c_1\cdot k(f)^m.\]
\end{thm}
\addtocounter{thm}{-1}
\endgroup

\begin{proof}
    Notice that it is sufficient to prove the theorem when $M$ is connected, since we can just sum the bound obtained componentwise. Therefore let us assume $M$ connected.    
    By the Nash--Tognoli Theorem (\Cref{Nash tognoli theorem}) and \cref{rem:nash tognoli}, there exists $\widetildeto{A}{M}\subseteq \R^n$  algebraic, regular, irreducible variety diffeomorphic to $M$. Denote by $\psi\colon \widetildeto{A}{M}\to M $ the given diffeomorphism and by $\tilde{g}$ the metric induced by $i\colon \widetildeto{A}{M}\to \R^n$.
    Consider $\Tilde{f}\coloneqq f \circ \psi \colon \widetildeto{A}{M}\to \R^k$. 
    
    By \Cref{weierstrass appro to use} there exists $c_0=c_0(\widetildeto{A}{M})$ such that, for every $d\in \mathbb{N}$, there exists a polynomial map $p=(p_1,\dots,p_k)$, where $p_i\in P_{n,d}$, such that 
    \[\|\tilde{f}-p\|_{C^0(\widetildeto{A}{M},\tilde{g})} \le \frac{c_0}{d}\|\tilde{f}\|_{C^1(\widetildeto{A}{M},\tilde{g})}.\]
    Let us consider $\delta_{\tilde{g}}(\tilde{f}) $ and the condition  number 
    \[k_{\tilde{g}}(\tilde{f})=\frac{\|\tilde{f}\|_{C^1(M,\tilde{g})}}{\delta_{\tilde{g}}(\tilde{f})},\] where the subscript $\tilde{g}$ is a reminder of the metric we are considering. Now, choose $d\in \mathbb{N}$ such that 
    \begin{equation} \label{optimal degree}   2c_0k_{\tilde{g}}(\tilde{f})<d\le 2c_0k_{\tilde{g}}(\tilde{f})+1,
    \end{equation} we obtain
    $\|\tilde{f}-p\|_{C^0(\widetildeto{A}{M},\tilde{g})}<\frac{\delta_{\tilde{g}}(\tilde{f})}{2}.$

    Now we have the following chain of inequalities
    \begin{align*}    
        b(Z(f))&=b(Z(\tilde{f})) & \\
        &\le b(Z(p)\cap \widetildeto{A}{M}) &\text{(by \Cref{semicontinuity}),}\\    &=O(d^m)\phantom{\tilde{f}}&\text{(by \Cref{theorem B}),}\\
        &=O(k_{\tilde{g}}(\tilde{f})) &\text{(by \eqref{optimal degree}).}       
    \end{align*}

    It remains to prove that $k_{\tilde{g}}(\tilde{f})=O(k_g(f))$, with constants independent of $f$.
    This follows easily by the compactness of $\widetildeto{A}{M}$. Indeed, by compactness, the following quantities are finite: $$\lambda(M,g,\psi) \coloneqq \min_{(x,v)\in T^1 \widetildeto{A}{M}}\frac{1}{\|v\|_{\psi^*g}}, \;\;\; \Lambda(M,g,\psi)\coloneqq\max_{(x,v)\in T^1 \widetildeto{A}{M}}\frac{1}{\|v\|_{\psi^*g}}.$$ It follows that for each $(x,v) \in T\widetildeto{A}{M}$:
    \[\lambda\cdot \|v\|_{\psi^*g}\le \|v\|_{g_{\widetildeto{A}{M}}}\le \Lambda\cdot \|v\|_{\psi^*g}.\]
     Therefore the condition number $k_{\tilde{g}}(\Tilde{f})= \frac{\|\Tilde{f}\|_{C^1(\widetildeto{A}{M},\tilde{g})}}{\delta_{\tilde{g}}(\Tilde{f})}$ can be bounded by:
    \begin{equation}\label{final equation theorem B}
    \frac{\lambda\cdot \|\tilde{f}\|_{C^1(\widetildeto{A}{M},\psi^*g)}}{\Lambda \cdot \delta_{\psi^*g}(\tilde{f})}\le\frac{\|\Tilde{f}\|_{C^1(\widetildeto{A}{M},\tilde{g})}}{\delta_{\tilde{g}}(\Tilde{f})}\le \frac{\Lambda\cdot \|\tilde{f}\|_{C^1(\widetildeto{A}{M},\psi^*g)}}{\lambda \cdot \delta_{\psi^*g}(\tilde{f})}.\end{equation}    
    Since $\psi:(M,g)\to (\widetildeto{A}{M},\psi^*g)$ is an isometry then $k_{\psi^*g}(\Tilde{f})=k_g(f)$. Therefore \eqref{final equation theorem B} implies that $k_{\tilde{g}}(\tilde{f})=O(k_g(f))$ which concludes the proof. 
\end{proof}
\begin{remark}
    In the previous theorem, the constant $c_1(M,g)$ a priori also depends on the Nash--Tognoli diffeomorphism $\psi:\widetildeto{A}{M}\to M$ that we choose. Therefore, to be more precise, we should write $c_1(M,g)=c_1(M,g,\psi)$. However, since no object in \Cref{theorem A} depends on $\psi$, we can consider $$c_1(M,g)\coloneqq\inf_{\psi:\widetildeto{A}{M}\to M}c_1(M,g,\psi)$$ which makes the constant $c_1$ independent of $\psi$.
\end{remark}
\subsection{Sharpness of the bound}
The next proposition shows that the bound obtained in \Cref{theorem B} can be attained, up to some constants which are independent of $f$. 
\begin{prop}\label{propo:sharp}
    Let $(M,g)$ be a smooth Riemannian manifold of dimension $m$.
    There exists a bounded sequence $\{f_n\}_{n\in \mathbb{N}}\subset C^1(M,\R^k)$ with
    \[\lim_{n\to \infty} k(f_n)=+\infty\]
    and a constant $c_2=c_2(m)>0$ such that for every $n\in \mathbb{N}$ the zero set $Z(f_n)$ is regular and 
    \[b(Z(f_n))\ge c_2\cdot k(f_n)^m\]
\end{prop}
\begin{proof} 
    Consider a smooth function $\psi \colon \D^m \to \R^k$ such that 
 $\psi|_{\{\|x\|>\frac{1}{2}\}}\equiv 0$ and such that it attains a non--zero regular value $a\in \R^k$ for some $x\in \D^m$. Set $g\coloneqq a-\psi$. Since $Z(g)=\psi^{-1}(a)\hookrightarrow \D^m,$ then the zero set of $g$ is non empty smooth compact submanifold contained in the interior of $\D^m$. Therefore \begin{equation}\label{sharpness betti number}
     b(Z(g))\ge 2,
 \end{equation} and since  $0$ is a regular value of $g$, we have also that  $k(g)<\infty$.
     For each $n\in \mathbb{N}$, $n>0$, let $I_n$ be a set of points $x_{n,i}\in \D^m$ such that the disks $D_{n,i}$ centered at $x_{n,i}$ with radius $n^{-1}$ are disjoint.
     Since the Hausdorff dimension of $\D^m$ is $m$, we can assume that the number of points is  \begin{equation}\label{number of balls}
         |I_n|=C\cdot n^m
     \end{equation}
     
     Define the sequence of functions $(g_n)_{n\in \mathbb{N}}\subset C^1(\D^m,\R^k)$ by  \[    g_n(x) \coloneqq     \begin{cases}        g\left(n(x - x_{n,i})\right) & \text{if } x \in D_{n,i} \text{ for some } i \in I_n, \\        a & \text{otherwise}.    \end{cases}\]
     Each $g_n$ is well defined and smooth because $f$ is constant equal to $a$ on a neighborhood of the boundary of each disk $D_{n,i}$. Moreover  $Z(g_n)$ is the disjoint union of $|I_n|$ submanifold diffeomorphic to $Z(g)$. 
     We also have that 
     \begin{align*}
     \delta(g_n)&=\inf_{x\in \D^m} \|g_n(x)\|+\sigma_k( D_xg_n(x))\ge \delta(g) \\
     \|g_n\|_{C^1}&= \sup_{x\in \D^m} \|g_n(x)\|+\sigma_k( D_xg_n(x))\le  n\|g\|_{C^1},\end{align*}
     which implies that \begin{equation}\label{condition number fn}
         k(g_n)\le n k(g).
     \end{equation}   
     From \eqref{sharpness betti number}, \eqref{number of balls} and \eqref{condition number fn}  , we obtain
     \[b(Z(g_n))=2Cn^m\ge 2C\left(\frac{k(g_n)}{k(g)}\right)^m.\]
     Let now $\phi \colon U \to \D^m$ be a local chart of $M$ and set $f_n\coloneqq g_n\circ \phi$ and $f\coloneqq g\circ \phi$. Since $f_n$ and $f$ are constant equal to $a$ at the boundary of $U$, we can extend them in a smooth way by setting $f_n(x)=a$ and $f(x)=a $ for any $x\in M\setminus U$. Choosing on the disk the metric $\phi^{-1}g_M^*$, the condition numbers of $f_n$ and $g_n$ coincide. Therefore, setting $c_2=2C k(f)^{-m}$, we obtain
     \[b(Z(f_n))\ge c_2\cdot k(f_n)^m.\]
     
\end{proof}

\subsection{Constant dependence}
In this section, we will show that the constant $c_1=c_1(M,g)$, which appears in \Cref{theorem B}, can be chosen so that it does not  depend on a specific manifold, but only on certain quantities that we will analyze. In particular, we will maximize $c_1(M,g)$ over a class of Riemannian manifolds $\mathcal{M}(m,D,\Lambda,V)$. Because this class is compact in the $C^1$ topology, we obtain a finite constant $c_3=c_3(m,D,\Lambda,V)$ depending only on the family.

\begin{definition}
    We denote by $\mathcal{M}(m,D,\Lambda,V)$ the set of smooth compact Riemannian manifolds $(M,g)$ of dimension $m$, diameter $\diam(M)\le D$, sectional curvature $|K_M|\le \Lambda$ and volume $\text{vol}(M)\ge V$.
\end{definition}
Notice that if $(M,g)\in \mathcal{M}(m,D,\Lambda,V)$,  since $\diam M< D$, then $M$ must be connected. This is not a big concern, since we can use the bound in \Cref{constant dependence} for each connected component that is contained in $\mathcal{M}(m,D,\Lambda,V)$.

The next theorem ensures that the class $\mathcal{M}(m,D,\Lambda,V)$ is compact in the $C^1$ topology. It can be found in \cite{Peters1987} Theorem 4.4.
\begin{theorem} \label{peters}
    Let $\{(M_k,g_k)\}_{k\in\mathbb{N}}\subseteq\mathcal{M}(m,D,\Lambda,V)$. Fixed $\alpha \in (0,1) $ there exists a subsequence $\{(M_h,g_h)\}$, a $C^{1+\alpha}$--Riemannian manifold $(M,g)$ and diffeomorphisms $\varphi_h \colon M \to M_h$ such that $(\varphi_h)^*g_h\to g$ in the $C^1$ topology.
\end{theorem}
As a consequence we can prove the existence of a uniform constant $c_3(m,\Lambda,D,V)$.
\begin{prop} \label{constant dependence} There exists a constant $c_3=c_3(m,\Lambda,D,V)>0$, such that for every  Riemannian manifold $(M,g)\in \mathcal{M}(m,D,\Lambda,V)$ and every $f\in C^1(M,\R^k)$, we have that
    \[b(Z(f))\le c_3 \cdot k(f)^m.\]
\end{prop}
\begin{proof}
    For any $m,D,V,\Lambda>0$, there are only finitely many diffeomorphism classes in $\mathcal{M}=\mathcal{M}(m,D,\Lambda,V)$. Indeed, suppose there were infinitely such classes. Then one could find a sequence $\{M_k\}_{k\in \mathbb{N}}\subset \mathcal{M}$ such that $M_{k_1}$ is not diffeomorphic to $ M_{k_2}$ for any $k_1\not=k_2$. By \Cref{peters}, however, there would exist a manifold $M$ and a subsequence $M_h$ such that each $M_h$ is diffeomorphic to $M$, contradicting the assumption that all $M_k$ are pairwise non--diffeomorphic.
    
    For each diffeomorphism class in $\mathcal{M}$, by Nash--Tognoli (\Cref{Nash tognoli theorem}), choose an algebraic representative $\widetildeto{A}{M}$, not necessarily belonging to $\mathcal{M}$. Given $(M,g)\in \mathcal{M}$, denote by
\[c_1(M,g)=\inf_{\psi'\colon \widetildeto{A}{M}\to M}c_1(M,g,\psi'),\]
where $\widetildeto{A}{M}$ is the algebraic representative of $M$ and $c_1(M,g,\psi')$ is the constant given by \cref{theorem A}.
Let now 
$$c_3=c_3(m,D,V,\lambda)\coloneqq\sup_{(M,g)\in \mathcal{M}}c_1(M,g).$$
By definition of sup and inf, for any $(M,g)\in \mathcal{M}$ and for any $\epsilon>0$, there exists a diffeomorphism $\psi$ satisfying:
\[c_1(M,g,\psi)< c_3+\epsilon.\]
Therefore $b(Z(f))\le (c_3+\epsilon)\cdot k(f)^m $, for any $f\colon M\to \R^k$, but since $\epsilon $ is arbitrary we have $$b(Z(f))\le c_3\cdot k(f)^m.$$
It remains to prove that  $c_3$ is finite.
Let $(M_k,g_k)\in \mathcal{M}$ be a maximizing sequence for $c_1(M,g)$, i.e. $\lim_k c_1(M_k,g_k)=c_3$.
By \Cref{peters}, there exists a subsequence $(M_h,g_h)$ and a $C^1$ Riemannian manifold $(M,g)$ such that $\varphi_h^*g_h \to g$ in the $C^1$ topology. 
Since $M_h,M$ belong to the same diffeomorphism class, they have a common algebraic representative $\widetildeto{A}{M}$ .

Let $\psi\colon \widetildeto{A}{M}\to M$ any diffeomorphism. The following chain of inequalities concludes the proof:
\begin{align*}
    c_3&=\lim_{h\to \infty} c_1(M_h,g_h)\\
    &\le \lim_{h\to \infty} c_1(M_h,g_h,\varphi_h\circ \psi)\\
    &= \lim_{h\to \infty} c_1(M,\varphi_h^*g_h,\psi)\\
    &=c_1(M,g,\psi).
\end{align*}
The third equality follows easily from the fact that $\varphi_h:(M,\varphi^*_h g_h) \to (M_h,g_h)$ is an isometry, and the last equality follows from the uniform convergence of the metric.
\end{proof}
\section{Semialgebraic--type sets}
\subsection{Definitions and basic properties}
Recall that a semialgebraic set described by a finite set of polynomials  $\{p_1,\dots, p_s\}\subseteq P_{n,d}$ can be written as
\[S\coloneqq \bigcup_i \bigcap_{j=1}^s \{p_j *_{ij}0\} ,\]
where $*_{ij}\in \{<,>,=\}$. We say that a semialgebraic set is described by closed conditions if $*_{ij}\in \{\le,\ge,=\}$. From a closed description, one can always obtain a non-closed description simply by splitting each inequality $\le$ into the cases $<$ and 
$=$.

In analogy with the classical semialgebraic case, if we replace the polynomials with smooth functions, one can consider:
\begin{definition}[Semialgebraic--type set]
     Let $\F=\{f_1,\dots,f_s\}$ be a smooth family (see \Cref{smooth family definition}). We say that a $S\subseteq M$ is a  semialgebraic--type set on a family $\F$ if it can be written in the form:
    \begin{equation}\label{closed semi}
    S\coloneqq \bigcup_i \bigcap_{j=1}^s \{f_j *_{ij}0\} , \end{equation}  
    where $*_{ij}\in \{<,>,=\}$.
\end{definition}

\begin{remark}\label{remark  closed semialgebraic type sets}
Notice that \eqref{closed semi} can be rewritten as $S=f^{-1}(T)$ with $T\subseteq \R^s$ given by:
\[T=\bigcup_i \bigcap_{j=1}^s \{x_j *_{ij}0\},\]
where $*_{ij}\in \{<,>,=\}$. The set $T$ is a union of strata of the orthant stratification $(\R^s,\mathcal{S})$. If $k(\F)<\infty$ then, by \Cref{equivalent description of regularity}, we have that $f\pitchfork \mathcal{S}$, so also $f\pitchfork T$. This implies that $S$ is Whitney stratified by the stratification $f^{-1}(\mathcal{S})$.
\end{remark}
The following lemma shows that, if a the family $\F$ is regular, then the closed semialgebraic type sets on the family $\F$ are always described by closed inequalities. Notice that without the regularity condition this property does not hold even for the classical semialgebraic sets. 
\begin{lemma}
    Let $S= \bigcup_i \bigcap_{j=1}^s \{f_j *_{ij}0\}$ a semialgebraic type set of a regular family $\F$.
    Then its closure is given by
    \[\overline{S}=\bigcup_i \bigcap_{j=1}^s \{f_j \;\overline{*}_{ij}\;0\}\]
    where $\overline{*}_{ij}$ is $\le,\ge$ or $=$ according as $*_{ij}$ is $<,>$ or $=$.
\end{lemma}
\begin{proof}
    Since closure commutes with unions, it suffices to show that $$\overline{\bigcap_{j=1}^s \{f_j *_j0\}}=\bigcap_{j=1}^s \{f_j \;\overline{*}_j\;0\}.$$
    The set $\bigcap_{j=1}^s \{f_j \;\overline{*}_j\;0\} $ is closed and clearly contains $\bigcap_{j=1}^s \{f_j *_j0\}$, so we have the inclusion:
    \[\overline{\bigcap_{j=1}^s \{f_j *_j0\}}\subseteq\bigcap_{j=1}^s \{f_j \;\overline{*}_j\;0\}.\] To prove the reverse inclusion we need to use the fact that the family $\F$ is regular. Let \begin{equation*} \label{eq:lemma closure}
    x\in  \bigcap_{j=1}^s \{f_j \;\overline{*}_j\;0\},\end{equation*} and denote by $J$ the set of indices $j=1,\dots,m$ such that  $f_j(x)=0$. 
    Consider the map $f_J\colon M \to \R^{|J|}$. Since  the family $\F$ is regular, $0$ is a regular value of $f_J$. Given that $x\in f_J^{-1}(0)$, this preimage is nonempty, and thus we have $|J|\le \dim M=m$. By the  rank theorem (see \cite[Theorem 4.12]{leeIntroductionSmoothManifolds2012}), there exists a coordinate chart $\varphi\colon \R^m \to U\subset M$ with $x=\varphi(0)$, satisfying:
    \[f_J\circ \varphi(y_1,\dots y_m)=(y_1,\dots, y_{|J|}).\]
    Choose $y=(y_1,\dots, y_m)$ such that $\sign y_j=*_j$, then $\varphi(y)\in \bigcap_{j\in J} \{f_j *_j0\}$. Moreover, since the relations $*_j$ for $j\not\in J$ are strict inequalities, choosing $y$  sufficiently close to $0$ we have also that $$\varphi(y)\in \bigcap_{j=1}^s \{f_j *_j0\}.$$
    As $\varphi(y)$ can be made arbitrarily close to  $x$, we conclude that
    $x\in \overline{\bigcap_{j=1}^s \{f_j *_j0\}}$ which completes the proof.
\end{proof}
In the following section we prove a bound for the Betti numbers of a closed semialgebraic--type set.  In the classical setting involving (polynomial) semialgebraic sets, see \cite[Theorem 7.38]{basuAlgorithmsRealAlgebraic2006}, the bound is polynomial in $s$ as well as in the degree of the polynomials. Therefore, also in our smooth setting, we search for a similar bound that is polynomial both in $s$ and in the condition number $k(\F)$ of the family.

\subsection{Cell decomposition of closed semialgebraic type sets}
The argument follows the classical proof for the bound of the Betti numbers of semialgebraic sets described by closed conditions, see \cite[Chapter 7.4]{basuAlgorithmsRealAlgebraic2006}. The strategy for the semialgebraic case is to perform a decomposition of the semialgebraic set into simpler pieces, i.e. basic semialgebraic sets. For basic semialgebraic sets, one can bound the Betti numbers using the classical Oleinik--Petrovsky--Thom--Milnor bound. 
Here we do something similar, we decompose the semialgebraic--type set into simple pieces  whose Betti numbers are controlled by \Cref{theorem A}.

We begin by introducing:
\[W_i\coloneqq \left\{\{f_i=0\},\{f_i =\frac{\delta}{2}\}, \{f_i=-\frac{\delta}{2}\}, \{f_i\ge \delta\}, \{f_i\le-\delta\}\right\},\] and
\[W_{\le i}\coloneqq\bigg\{\omega_i \mathrel{\Big|} \omega_i= \bigcap_{j=1}^i \psi_j,\; \psi_j\in W_j \bigg\}.\]
Notice that for each $\omega_i\in W_{\le i}$ and any $x\in \omega_i$, the sign of $f_j(x)$ is constant for each $j=1,\dots,i$.

\begin{figure}[ht]
    \centering
    \begin{subfigure}[t]{0.45\textwidth}
      \centering
      \fbox{%
        \includegraphics[width=\linewidth, keepaspectratio]{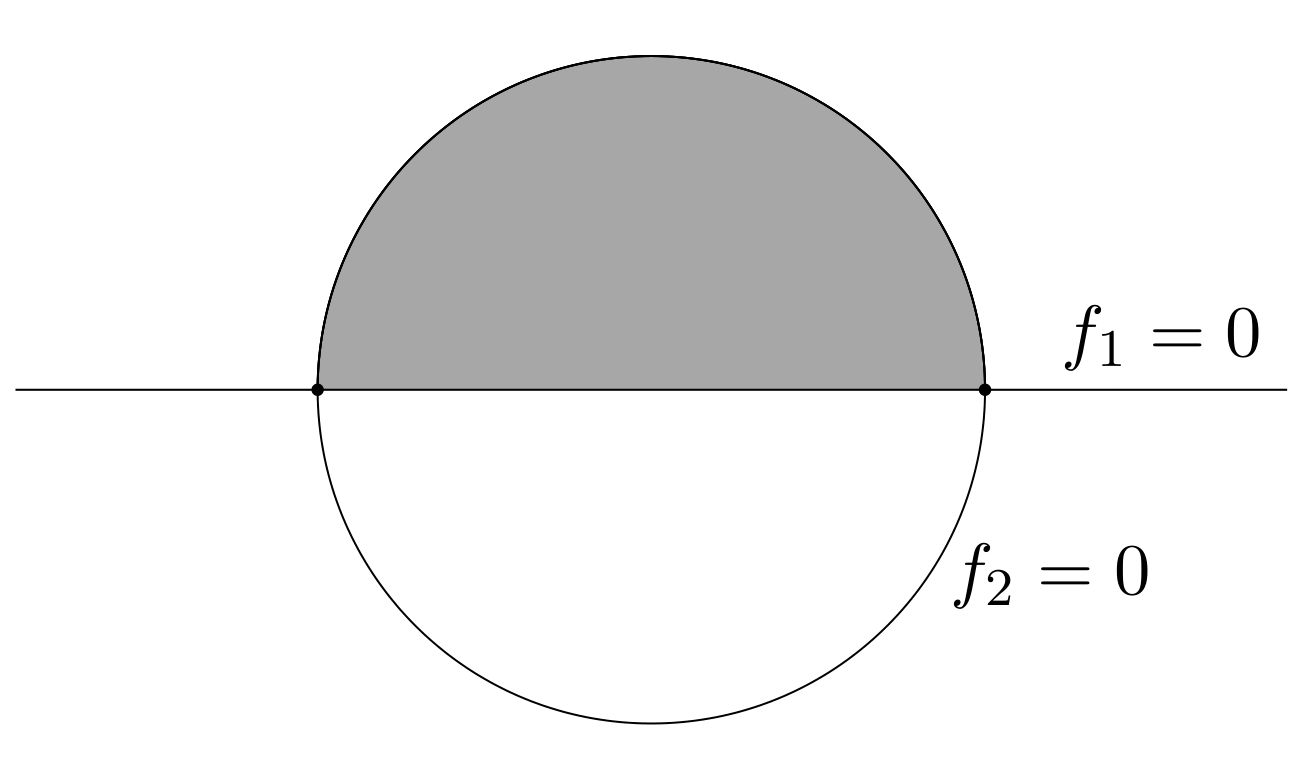}%
      }
      \caption[]{In gray we see a semialgebraic--type set $S$ on $
          f_1\coloneqq x_2,\;\; f_2\coloneqq x_{1}^{2}+x_{2}^{2}-1
      $}
    \end{subfigure}
    \hfill
    \begin{subfigure}[t]{0.45\textwidth}
      \centering
      \fbox{%
        \includegraphics[width=\linewidth, keepaspectratio]{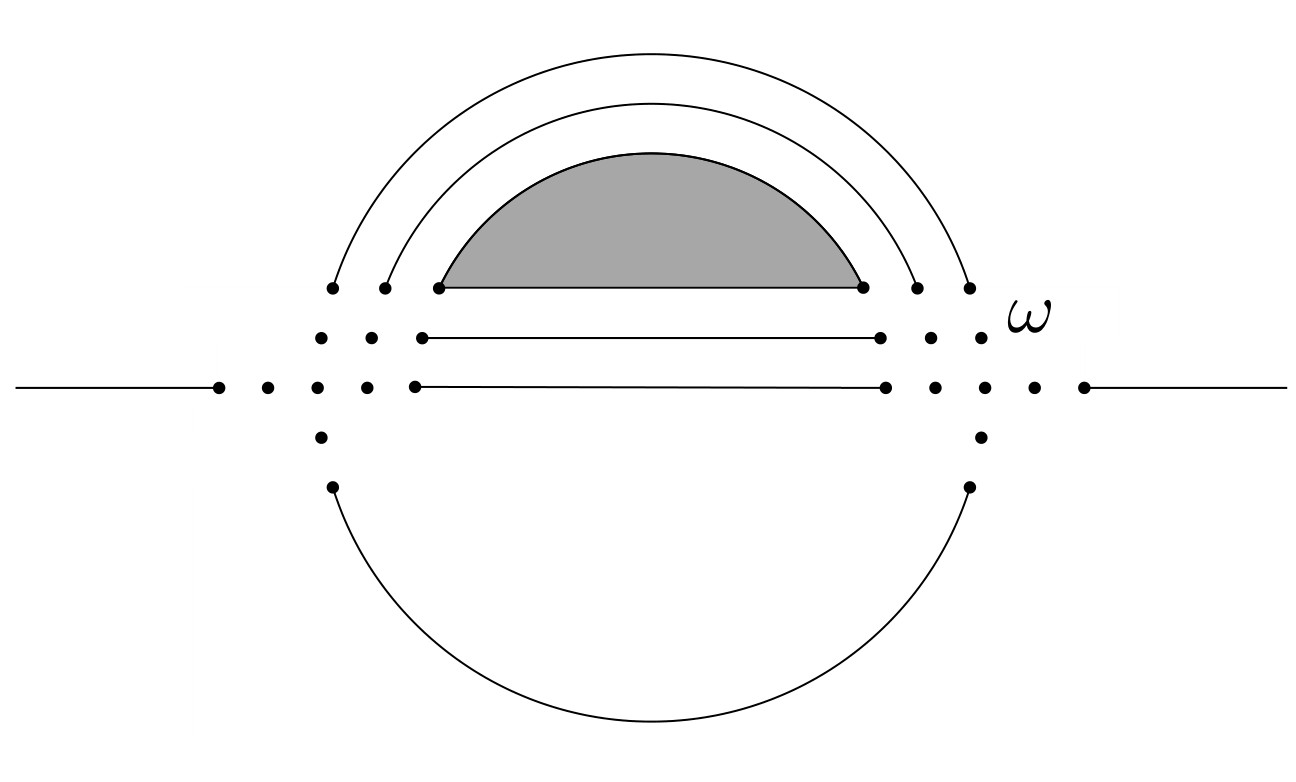}%
      }
      \caption{$\omega\in W_{\le 2}$ such that $\omega\subseteq S$. }
      \label{fig:right}
    \end{subfigure}
    
    \caption{}
    \label{fig:two_same_size}
\end{figure}

The next proposition relates the Betti numbers of $S$ to the Betti numbers of $\omega_s\in W_{\le s}$.
\begin{prop} \label{cells}
    Let $S$ be a closed semialgebraic--type set on a family $\F=\{f_1,\dots,f_s\}$. Suppose that $\delta(\F)>0$, then, for any $0<\delta<\sqrt{m+1}\cdot\delta(\F)$ we have that:
    \[b(S)\le \sum_{\substack{\omega_s\in W_{\le s} \\ \omega_s \subseteq S}}b(\omega_s).\]
\end{prop}

\begin{proof}
    We prove by induction that for every $i=1,\dots, s$
    $$b(S)\le \sum_{\omega_i\in W_{\le i} }b(S\cap\omega_i).$$
    If $i=0$ there is nothing to prove.

   Assume that the statement holds for some  $i < s$, namely $$b(S)\le \sum_{\omega_i\in W_{\le i}}b(S\cap \omega_i)$$ 
   We need to show it also holds for $i+1$:
   \[  b(S)\le \sum_{\omega_{i+1}\in W_{\le i+1}}b(S\cap \omega_{i+1}).\]
    By definition, $W_{\le i+1}=\{\omega_i \cap \psi \mid \omega_i\in W_{\le i}, \psi \in W_{i+1}\}$, thus, using the inductive hypothesis, it is sufficient to prove that for each fixed $\omega_i \in W_{\le i}$, $$b(S\cap \omega_i) \le \sum_{\psi\in W_{i+1}} b(S\cap \omega_i \cap \psi)$$
    Recall from the Mayer–Vietoris sequence, see \cref{section mayer vietoris}, that for any closed sets $A,B \subseteq M$, $$b(A\cup B)\le b(A)+b(B)+b(A\cap B).$$ Set $$A\coloneqq S\cap \omega_i \cap \{|f_{i+1}|\ge \delta\}, \quad B\coloneqq S\cap \omega_i \cap \{|f_{i+1}|\le \delta\}.$$ Then $A \cup B = S \cap \omega_i$ and, by the Mayer–Vietoris inequality,  
    \begin{align*}
    b(S\cap \omega_i)\le b(S\cap \omega_i \cap \{|f_{i+1}|\ge \delta\})+b(S\cap \omega_i \cap \{|f_{i+1}|\le \delta\})+b(S\cap \omega_i \cap \{|f_{i+1}|= \delta\}).
    \end{align*}
    Next, we claim the following homotopy equivalences: $$S\cap \omega_i \cap \{|f_{i+1}|\le \delta\} \simeq S\cap \omega_i \cap \{|f_{i+1}|=0\}$$ 
    and 
    \[S\cap \omega_i \cap \{|f_{i+1}|= \delta\} \simeq S\cap \omega_i \cap \{|f_{i+1}|=\delta/2\}.\]
    The sets $S\cap \omega_i \cap \{|f_{i+1}|\le \delta\}$ and $S\cap \omega_i \cap \{|f_{i+1}|= \delta\}$ are union of strata of the stratification $\S'_{\{1,\dots,\widehat{i+1},\dots,s\}}$ of $M$ given by the transversal intersections of the hypersurfaces
    \begin{align*}
    \{f_j=\pm \delta, f_j=\pm \delta/2, f_j=0, & \text{  for } j=1,\dots ,i;\\ f_j=0, & \text{ for } j=i+2,\dots,s \}.
    \end{align*}
    By \cref{prop: product stratification} and  \cref{rem: product stratification} applied to $I=\{i+1\}$, if $\delta_0<2^{-\frac{m}{2}}\delta(\F)$, there exists a stratum preserving homeomorphism $$\phi_{i+1}\colon f_{i+1}^{-1}((-\delta_0,\delta_0))\to(-\delta_0,\delta_0)\times f_{i+1}^{-1}(0)$$
    between the stratifications $f^{-1}(\S_{i+i}\times\S'_{\{1,\dots,\widehat{i+1},\dots,s\}} )$ and $\S_{i+i}\times  f^{-1}(\S'_{\{1,\dots,\widehat{i+1},\dots,s\}} )$.
    Now, consider a continuous non decreasing map $\alpha: (-\delta_0,\delta_0) \to (-\delta_0,\delta_0)$ such that $\alpha([-\delta,\delta])=0$ and which is the identity near $\partial(-\delta_0,\delta_0)$. For any cell $\sigma$ in the stratification $\S_{i+1}$ of $(-\delta,\delta)$, i.e. $\sigma\in\{(-\delta_0,0),\{0\},(0,\delta_0)\}$, we have that $\alpha(\sigma)\subseteq \overline{\sigma}$. Moreover $\alpha|_{\overline{\sigma}}\colon \overline{\sigma} \to \overline{\sigma}$ is an homotopy equivalence.  This implies that also:
    \[\Phi\coloneqq\phi_{i+1}^{-1}\circ (\alpha\times \id_{f_{i+1}^{-1}(0)})\circ \phi_{i+1}\colon f_{i+1}^{-1}((-\delta_0,\delta_0))\to f_{i+1}^{-1}((-\delta_0,\delta_0))\]
    is a homotopy equivalence which extends to the identity outside $f_{i+1}^{-1}((-\delta_0,\delta_0))$ and such that $\Phi(\overline{\sigma})\subseteq \overline{\sigma}$ for each cell in the stratification.
    Therefore we can restrict $\Phi$ to  $S\cap \omega_i \cap \{|f_{i+1}|\le \delta\}$ obtaining again an homotopy equivalence. In particular proves that 
    $$S\cap \omega_i \cap \{|f_{i+1}|\le \delta\} \simeq S\cap \omega_i \cap \{|f_{i+1}|=0\}$$
    By an analogous reasoning we can prove that \[S\cap \omega_i \cap \{f_{i+1}=\pm \delta\} \simeq S\cap \omega_i \cap \{f_{i+1}=\pm\delta/2\}.\]
    This shows that $b(S\cap \omega_i) \le \sum_{\psi\in \Sigma_{i+1}} b(S\cap \omega_i \cap \psi)$ and concludes the inductive step.
    
    Finally, since $S$ is described by a set of closed sign conditions and  each of them is constant on each $\omega_s$, then whenever $S \cap \omega_s \neq \emptyset$ we actually have $ \omega_s \subseteq S$, therefore we obtain
    \[\sum_{\omega_s\in W_{\le s} }b(S\cap\omega_s)=\sum_{\substack{\omega_s\in W_{\le s} \\ \omega_s \subseteq S}}b(\omega_s),\]
    which concludes the proof.
\end{proof}
\subsection{Betti numbers of closed semialgebraic type set}
Before proving \cref{theorem C}, we states the following lemma. It shows that boundaries of semialgebraic-type sets have strictly lower dimension than the sets themselves and allow us to apply \cref{Mayer Vietoris}

\begin{lemma} \label{boundary dimension}
    Let $k(\F)<+\infty$ and let $S$ be a semialgebraic--type set on the family $\F$. Then
    \[\dim (\partial S)< \dim ( S).\]
\end{lemma}
\begin{proof}   
    By \Cref{remark  closed semialgebraic type sets}, we have that $S=f^{-1}(T)$, where $T$ is a semialgebraic set and $f\pitchfork T$. Since $f$ is continuous, then $\partial f^{-1}(T) \subseteq f^{-1}(\partial T)$, therefore
    \[\codim (\partial S)\ge \codim f^{-1}(\partial T).\]
    Since $f\pitchfork T$, we have that $\codim f^{-1}(\partial T)=\codim \partial T$. By Theorem 5.42 in \cite{basuAlgorithmsRealAlgebraic2006}, for a semialgebraic set $T$ we have
    $\codim \partial T>\codim T.$
    In conclusion we obtain
    \begin{align*}
        \codim \partial S&\ge \codim f^{-1}(\partial T) \\
        &=\codim\partial T \\
        &>\codim T \\
        &=\codim S
    \end{align*}
    we the last equality follows again by transversality.
\end{proof}

Before the proof of \cref{theorem C} we recall the statement.
\begingroup
\def\thethm{\ref{theorem C}}
\begin{thm} 
    Let $(M,g)$ be a smooth Riemannian manifold of dimension $m$. There exists a constant $c_4(M,g)>0$ such that for every smooth family $\F$, and every closed semialgebraic--type set $S$ on the family $\F$, we have  
    \[b(S)\le c_4 \big(s\cdot k(\F)\big)^m .\]
\end{thm}
\addtocounter{thm}{-1}
\endgroup
At this point, one could obtain an easy bound for the Betti numbers by computing the Betti number of each cells. Although $b(\omega_s)=O(k(\F)^m)$ and summing over all $\omega_s \in W_{\le s}$ we would obtain  $b(S)=O(k(\F)^m\cdot 5^s)$, which is exponential in $s$.
\begin{proof}
By \Cref{cells}, it remains to prove that
\begin{equation}\label{eq:theorem C}
\sum_{\omega_s\in W_{\le s}}b(\omega_s)\le c_4(s\cdot k(\F))^m.
\end{equation}

To obtain a bound that is polynomial in $s$ we need an additional construction, which can be found in \cite[Chapter 7.4]{basuAlgorithmsRealAlgebraic2006}. In particular define, for $1\le i\le s$, 
\[
    \Omega_j\coloneqq \{f_j(x)\le -\delta\}\cup \{f_j(x)\ge \delta\}\cup  
    \{f_j(x)= \pm \frac{\delta}{2}\}\cup\{f_j(x)=0\} \]
    and  set $$\Omega \coloneqq \bigcap_{j=1}^s \Omega_j.$$ Since $\Omega$ is the disjoint union of all cells $\omega_s\in W_{\le s}$, it follows that
    $$b(\Omega)=\sum_{\omega_s\in W_{\le s}}b(\omega_s).$$ Since by \cref{boundary dimension}
      $b_m(\partial \Omega_j)=0$, we can apply  \Cref{Mayer Vietoris} obtaining :
     \begin{align} \label{eq:Theorem C 2} b_i(\Omega)=b_i\Big(\bigcap_{j=1}^s\Omega_j\Big)\le& \sum_{h=1}^{m-i}\sum_{\substack{H\subseteq \{1,\dots, s\}\\|H|=h}}b_{i+h-1}\Big(\bigcup_{j\in H} \Omega_j \Big)+\binom{s}{m-i}b_m(M)
     \end{align}
    We conclude the proof in the next lemma by showing that if $i+h-1< m$ then $$b_{i+h-1}\Big(\bigcup_{j\in H} \Omega_j \Big)\le (6^h-1)c_1\cdot k(\F)^m+b_{i+h-1}(M).$$
    In this way \eqref{eq:Theorem C 2} becomes:
    \[b_i(\Omega)\le \sum_{h=1}^{m-i}\binom{s}{h}\big((6^h-1)c_1\cdot k(\F)^m+b_{i+h-1}(M)\big)+\binom{s}{m-i}b_m(M),\]
    or more compactly, $b_i(\Omega)=O(s^{m-i}k(\F)^m )$. In particular we have that $b(\Omega)=O(s^{m}k(\F)^m )$, which  concludes the proof of the theorem.
\end{proof}
It remains to show the following lemma.
\begin{lemma} \label{X_h bound}
For each $0\le i\le m$ and each $H\subseteq \{1,\dots, s\}$, such that $|H|=h$, we have:
 \[ b_i(\bigcup_{j\in H} \Omega_j )\le (6^h-1)c_1\cdot k(\F)^m+b_i(M).\]
 \end{lemma}
\begin{proof}
    For each $H\subseteq \{1,\dots, s\}$, such that $|H|=h$, consider $$X_j\coloneqq\{f_j=\pm \delta\}\cup\{f_j=\pm \frac{\delta}{2}\}\cup\{f_j=0\},$$
    and let \[X_H=\bigcup_{j\in H}X_j.\]
    By \Cref{Mayer Vietoris}, the Betti number $b_i(X_H)$ can be bounded above by the sum of Betti numbers of all $\ell$--ary intersections of the $X_j$'s, for $1\le \ell \le h$.
    The number of $\ell$--ary intersections is $\binom{h}{\ell}$. Each intersection is a disjoint union of $5^\ell$ sets. By \Cref{thom isotopy particular}, if $\delta<\delta(\F)$, each of these intersections is homotopy equivalent to  \[\{f_i=0 \mid i\in L\},\]
    for some $L\subseteq H$.
    Since these sets are smooth preimages of $0$, by \cref{theorem A} their Betti numbers are bounded above by $c_1k(\F)^m$. Summing over all possible $\ell$--ary intersections, 
    \begin{equation} \label{eq: X_H}
    b_i(X_H)\le \sum_{\ell=1}^h 5^\ell \binom{h}{\ell}c_1\cdot k(\F)^m=(6^h-1)c_1\cdot k(\F)^m. \end{equation}
    
    Let now $$F=\left(\bigcap_{i=1}^h\{\abs{f_i}\le\delta\}\right)\cup X_h.$$ By the Mayer--Vietoris inequality applied to $\bigcup_{j\in H} \Omega_j$ and $F$,
    \[b_i(\bigcup_{j\in H} \Omega_j)\le b_i(\bigcup_{j\in H} \Omega_j\cap F)+b_i(\bigcup_{j\in H} \Omega_j\cup F)=b_i(X_H)+b_i(M).\]
    By \eqref{eq: X_H} we obtain the desired result:
    \[b_i(\bigcup_{j\in H} \Omega_j)\le (6^h-1)c_1(k(\F))^m+b_i(M).\]
\end{proof}

\newpage
\bibliography{bibliography}
\bibliographystyle{acm}
\end{document}